\documentclass[11pt,twoside]{amsart}
\title{Regular del Pezzo surfaces with irregularity}
\author{Zachary Maddock}
\date{April 18, 2013}

\input xy %This is a package for doing commutative diagrams
\xyoption{all}

\usepackage{latexsym}
\usepackage{amssymb}
\usepackage{amsfonts}
\usepackage{amstext}
\usepackage{amsmath}
\usepackage{times}
\usepackage{graphicx}
\usepackage{enumerate}
\usepackage{url}
\usepackage{blkarray}
\usepackage{multirow}
\usepackage{cite}
\usepackage{hyperref}

\DeclareFontFamily{OMS}{rsfs}{\skewchar\font'60}
\DeclareFontShape{OMS}{rsfs}{m}{n}{<-5>rsfs5 <5-7>rsfs7 <7->rsfs10 }{}
\DeclareSymbolFont{rsfs}{OMS}{rsfs}{m}{n}
\DeclareSymbolFontAlphabet{\scr}{rsfs}

 \renewcommand{\theequation}{\arabic{section}.\arabic{subsection}.\arabic{equation}}

\newcommand{\Q}{\mathbb{Q}}
\newcommand{\G}{\mathbb{G}}

\newcommand{\A}{\mathbb{A}}
\newcommand{\F}{\mathbb{F}}
\renewcommand{\P}{\mathbb{P}}

\newcommand{\im}{\imagecommand}

\newcommand{\eps}{\varepsilon}
\newcommand{\inv}{^{-1}}

\newcommand{\surject}{\twoheadrightarrow}

\newcommand{\OO}{\mbox{$\mathcal{O}$}}

\newcommand{\sB}{\mathcal{B}}

\newcommand{\sE}{\scr{E}}
\newcommand{\sF}{\scr{F}}

\newcommand{\sI}{\scr{I}}

\newcommand{\sL}{\scr{L}}

\newcommand{\sX}{\mathcal{X}}

\newcommand{\sZ}{\mathcal{Z}}

\DeclareMathOperator{\Hom}{Hom}

\DeclareMathOperator{\Spec}{Spec}

\DeclareMathOperator{\imagecommand}{Im}

\newcounter{repeatcounter}

\newtheorem{theorem}[equation]{Theorem}

\newtheorem{lemma}[equation]{Lemma}
\newtheorem{proposition}[equation]{Proposition}

\newtheorem{corollary}[equation]{Corollary}
\newtheorem{remark}[equation]{Remark}

\newtheorem{question}[equation]{Question}

\theoremstyle{definition}
\newtheorem{definition}[equation]{Definition}
\newtheorem{definition/lemma}[equation]{Definition/Lemma}

\input xy %This is a package for doing commutative diagrams
\xyoption{all}

\newcommand{\comment}[1]{}

\DeclareMathOperator{\Sym}{Sym}

\DeclareMathOperator{\td}{td}
\DeclareMathOperator{\ch}{ch}

\newcommand{\red}{\textrm{red}}
\newcommand{\sing}{\textrm{sing}}
\newcommand{\ex}{\mathcal X}
\newcommand{\zee}{\mathcal Z}
\newcommand{\Frobenius}{\mathbf F}

\begin{document}

\begin{abstract}
  We construct the first examples of regular del Pezzo
  surfaces $X$ for which \mbox{$h^1(X, \OO_X) > 0$.}  We also find a
  restriction 
  on the integer pairs that are possible as the anti-canonical degree $K_X^2$
  and irregularity $h^1(X,\OO_X)$ of such a surface.
\end{abstract}

\maketitle

\setcounter{tocdepth}{1}
{
\tableofcontents
}

\setcounter{section}{0}
\section*{Introduction}

\subsection{Regular varieties}
Any variety defined over a finitely generated extension $k$ of a
perfect (e.g. algebraically closed) field $\F$ can be viewed 
as the generic fibre of a morphism of $\F$-varieties  $\sX \to \sB$
 such that $k$ is the function field of the base $\sB$.
In this way, the geometry of varieties
 over imperfect fields 
is relevant to 
the understanding of
 the birational geometry of varieties
over algebraically closed fields of positive
characteristic.  
One main difficulty that arises
is that, unlike over perfect fields,
the notions of smoothness and regularity diverge: a smooth variety is
necessarily regular, but a regular variety may not be smooth.

\begin{definition}
 A
variety $X$   
is defined to be \emph{regular} provided that the local coordinate
ring $\OO_{X,x}$ 
is a regular local ring at all points $x \in X$.
  A $k$-variety $X$ is 
\emph{smooth} over $k$ provided that it is
geometrically regular (recalling that a $k$-variety $X$ is said to
satisfy a property \emph{geometrically} if the base change $X_{\bar
  k}$ to the algebraic closure satisfies the given
property) . 
\end{definition}

The notion of smoothness is well-behaved, due largely 
to the fact that a $k$-variety $X$ is smooth if and only if the
cotangent sheaf $\Omega_{X/k}$ is a vector bundle of rank equal to the
dimension of $X$.
Regularity, like smoothness, is a local
property, and can be described in terms of the latter as follows:
 a $k$-variety $X$ is regular if and only if 
there exists a smooth $\F$-variety $\sX$ 
and a morphism $\sX \to \sB$ of which $X$ is
the generic fibre.  
In characteristic $0$, a general fibre of a morphism between smooth
varieties is smooth, yet
in positive characteristic it is common for such morphisms 
to admit no smooth fibres.  In fact,
the collection of generic fibres of 
morphisms between smooth $\F$-varieties
that admit no smooth fibres
precisely comprises the
regular, non-smooth varieties over finitely generated field extensions
of $\F$.
A standard example 
is the generic
fibre of the family $(y^3 = x^2 + t) \subseteq \A^2 \times
\A^1$
 of cuspidal plane curves, parameterized
by the affine coordinate $t$ over a field of
characteristic $2$.

\subsection{New results}
Our study focuses on regular del Pezzo surfaces, a class of 
varieties that, 
as we discuss in \S\ref{subsection-minimal-model}, arises
naturally in the context of the minimal model program. 
\begin{definition}\label{defn-del-Pezzo}
  A \emph{del Pezzo scheme} over a field $k$ is defined to be a
  2-dimensional, projective, Gorenstein scheme $X$ of finite-type over
  $k = H^0(X,\OO_X)$ which is \emph{Fano}, that is, for which
  the inverse of the dualizing sheaf, $\omega_X\inv$, is an ample line
  bundle.  A \emph{del Pezzo surface} is a del Pezzo
  scheme that is an integral scheme.
\end{definition}
This paper answers affirmatively the question of whether there exist
regular del Pezzo surfaces $X$ 
 that are geometrically non-normal or
geometrically non-reduced
by constructing examples of each type which have positive irregularity
$h^1(X, \OO_X) = 1$.  
We also find a characteristic-dependent restriction on the
anti-canonical degree of regular del 
Pezzo surfaces that have a given positive irregularity $q := h^1(X, \OO_X) > 0$.
The main result (represented graphically in
Figure~\ref{graph-figure}) can be concisely summarized as follows:\\

\noindent \textbf{Main Theorem.}
  \begin{enumerate}
  \item \emph{ There exist regular del Pezzo surfaces, $X_1$ and
  $X_2$, with irregularity $h^1(X_i, \OO_{X_i}) = 1$ and of degrees
    $K_{X_1}^2 = 1$ and 
    $K_{X_2}^2 = 2$.  The surface $X_1$ is
  geometrically integral and defined over the field
  $\F_2(\alpha_0, \alpha_1, \alpha_2)$ while $X_2$ is geometrically
  non-reduced and defined over the index-$2$ subfield 
  $\F_2(\alpha_i \alpha_j: 0 \leq i,j \leq 3) \subseteq \F_2(\alpha_0,
  \alpha_1,  \alpha_2, \alpha_3)$.} \label{item-main-theorem-construction}  
  \item  \emph{If $X$ is a normal, local complete intersection (l.c.i.) del
    Pezzo surface (e.g. a regular del Pezzo surface)  
  with irregularity $q > 0$ and anti-canonical degree $d = K_X^2$ over a field of
  characteristic $p$, then} 
  \begin{equation}\label{inequality}
    q \geq \frac {d(p^2-1)}{6}. 
  \end{equation} \label{item-main-theorem-bound}
  \end{enumerate}

\begin{figure}[h]
  \includegraphics[width = 20em]{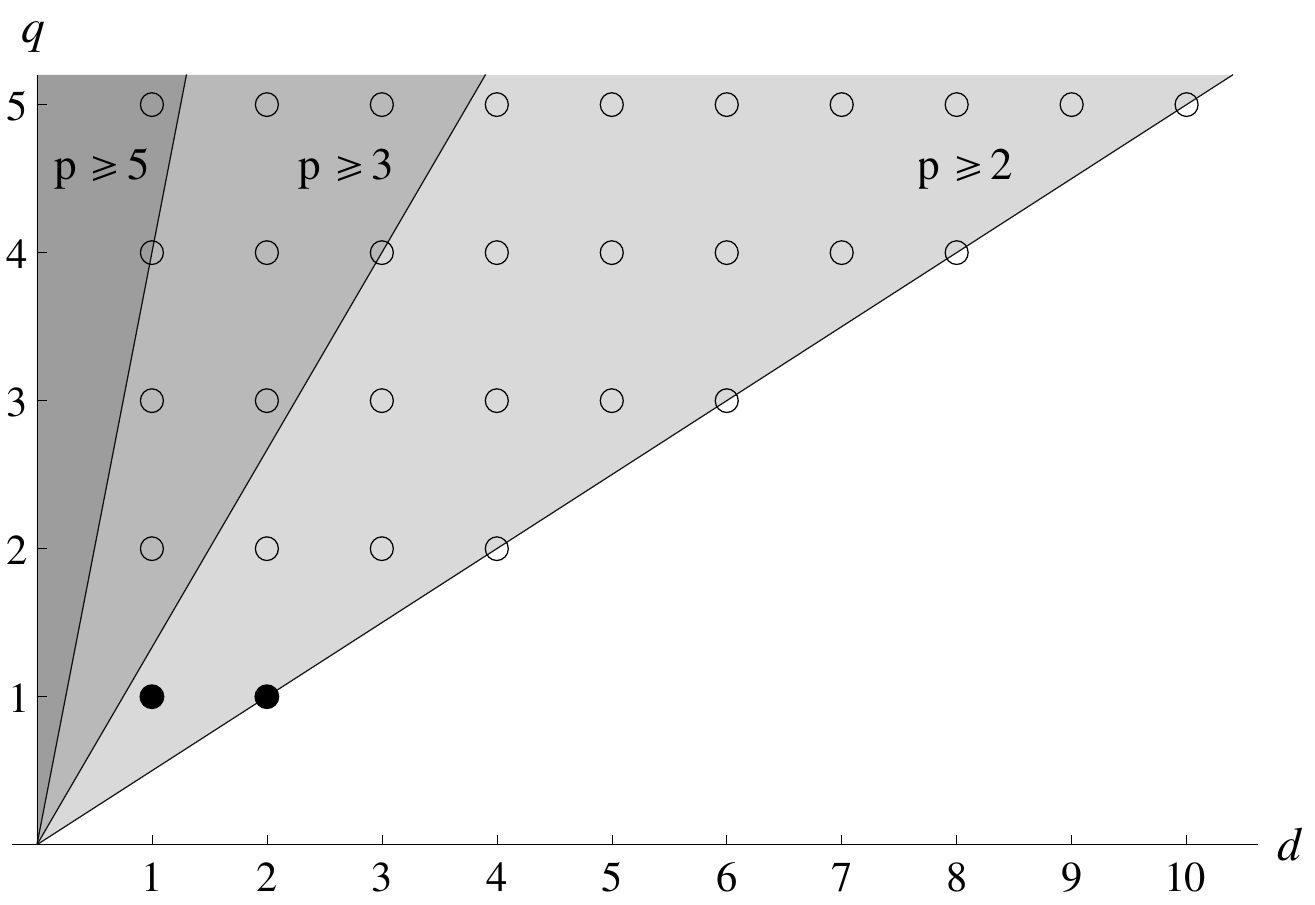} 
  \caption{Circles represent possible values for the
    degree $d$ and irregularity $q$ of an l.c.i.~and normal del Pezzo
    surface with positive irregularity $q > 0$.
    Solid dots represent actual values attained by
    the regular del Pezzo surfaces constructed in
    \S\ref{subsection-q-1-d-1} and \S\ref{subsection-q-1-d-2}.  
    Shaded regions demonstrate how the inequality \eqref{inequality}
    becomes more restrictive as the characteristic grows.}  
  \label{graph-figure}
\end{figure}

As our proof of \eqref{item-main-theorem-construction} is
constructive, 
it should be possible to obtain
concrete descriptions of the geometry in each example.
We do so
for the degree one surface $X_1$,
proving by explicit computation in Proposition
\ref{prop-explicit-description-of-reid} 
a detailed version of the following proposition.\\

\noindent \textbf{Proposition.}
{ \it 
There exists a regular form, $Z$, of a double plane in
$\P^3$ and a finite, inseparable morphism $f: Z \to X_1$ of degree $p
= 2$.  Moreover, if $\bar Z$ and $\bar X_1$ denote the geometric base
changes of $Z$ and $X_1$, respectively, then this
construction has the following properties:
\begin{enumerate}
\item The induced morphism $f^{\red}: \P^2 \cong \bar Z^{\red} \to \bar X_1$ is the
  normalization of $\bar X_1$.
\item The singular locus of $\bar X_1$ is a rational cuspidal curve
  $C$ of arithmetic genus one.
\item The inverse image of $C$ under $f^{\red}$ is a non-reduced double line in $\P^2$.
\end{enumerate}
}

\subsection{Motivation from the minimal model program}\label{subsection-minimal-model}
Among the varieties over function fields, Fano varieties such as del
Pezzo surfaces are
of particular interest, due to their
prominent role in the minimal model program.
In brief,
the goal of the program
is to understand the birational geometry of a variety $X$ by 
constructing a birational model
$\hat X$ whose canonical divisor $K_{\hat X}$ is a nef divisor;
one calls such a variety $\hat X$ a \emph{minimal model} of $X$.
If $\hat X$ is smooth, then the terminology is justified: 
$\hat X$ is minimal in the sense that any birational morphism
$\hat X \to X'$ to a smooth variety $X'$ is an
isomorphism (cf. \cite[Prop.~1.45]{deb}).  

If $X$ is not itself a minimal model, then there exist effective curves 
$C \subseteq X$ that pair negatively
with the canonical divisor, $C.K_X < 0$.  The strategy for
constructing $\hat X$ is to 
attempt to contract precisely 
these negative curves
 via birational morphisms $f:X \to Y$ and
then to partially 
resolve any serious singularities that were introduced.
However, the contraction morphism
associated to a certain negative curve
may not be birational, and
the contracted variety $Y$ may be of lower dimension, as is the case, say,
for ruled surfaces.
Since the curves 
contained in fibres of $f$
each pair negatively
with $K_X$,  the fibres of $f$ are therefore Fano schemes by
Kleiman's criterion for ampleness.
In other words, the contraction morphism $f:X \to Y$ realizes $X$ as a
Fano fibre-space.   

When $X$ is a smooth $3$-fold over an
algebraically closed field,
theorems of Mori \cite{mor} and
Koll\'ar 
\cite{kol0} guarantee
that any given extremal ray in the
cone of effective curves pairing negatively with $K_X$ can be
contracted by a morphism $f:X \to Y$ to a normal variety $Y$.
Furthermore, they classify these contraction morphisms:
either $f$ is birational, equal to the inverse
of the blowing-up of a  
point or a smooth curve in $Y$, or
$f:X \to Y$ is a Fano fibration over a smooth variety $Y$ of dimension
at most $2$.
If $Y$ is a point, then $X$ is itself a Fano $3$-fold, while
if $Y$ is a surface, then $X$ is a 
 conic bundle over $Y$.  

Our case of interest
is when $Y$ is
a curve, as then $f:X \to Y$ is a del Pezzo fibration.  
Since $X$ is smooth, the generic fibre of the fibration is a regular
del Pezzo surface over the function field of $Y$.
In characteristic $0$, regular del Pezzo surfaces are smooth,
and there are some results toward a birational
classification of these del Pezzo fibrations 
(cf. \cite{gri} for a recent survey).
In positive characteristic, however, 
 the generic del Pezzo surface is potentially non-smooth,
and  the situation is not so clear.
Indeed, Koll\'ar asks whether these regular del Pezzo surfaces can be
geometrically non-normal, 
or even geometrically non-reduced, but remarks that understanding
this phenomenon seems complicated, especially in characteristic
$2$ (cf. \cite[Rem.~1.2]{kol0}).

\subsection{Regular forms and the classification of del Pezzo surfaces}

\label{subsection-classification-over-k-bar}
We can also contextualize our results in terms of the classification of del
Pezzo surfaces over an algebraically closed field.  In particular, we
will see how our Main Theorem makes progress toward determining
which singular (possibly non-normal or non-reduced) del Pezzo
schemes over algebraically closed fields admit regular $k$-forms for
some subfield $k$.
\begin{definition}
  Let $K/k$ be an extension of fields.  Given a $K$-variety $\bar X$, 
  one says that a $k$-variety $X$ is a \emph{($k$-)form} of $\bar X$
  provided that there exists an isomorphism $\bar X \cong X \times_k K$.
\end{definition}
We recall the classification of del Pezzo surfaces $X$ over an
algebraically closed field.
When $X$ is normal, 
Hidaka and Watanabe \cite{hid-wat1} prove that 
either $X$ is a rational surface
with singularities at worst rational double points or $X$ is a cone over an
elliptic curve.   Not all of these surfaces admit regular forms, as
Hirokado \cite{hir1} and Schr\"oer 
\cite{sch2} show how the existence of a regular form
 puts restrictions on the possible singularities.

In the course of proving the classification result, Hidaka and
Watanabe \cite{hid-wat1} prove that all normal del Pezzo
surfaces over an algebraically closed field satisfy $H^1(X, \OO_X) = 0$.
Over the complex numbers, this cohomological vanishing $H^1(X,\OO_X) = 0$
can be viewed as a consequence of 
the Kodaira vanishing theorem for normal surfaces (cf. \cite{mum1}), since
$H^1(X,\OO_X)$ is Serre-dual to the group $H^1(X, \omega_X)$ and
$\omega_X$ is the inverse of an ample line bundle.

Reid \cite{rei1} classifies the non-normal del Pezzo surfaces.
He shows that such surfaces $X$ are formed from rational, normal varieties 
$X^\nu$ by collapsing a (possibly non-smooth) conic
to a rational curve $C$ that is either smooth or has wildly cuspidal
singularities (i.e. cuspidal singular points of order divisible by the
prime characteristic $p > 0$).  
We remark that for these surfaces, the irregularity is equal to the
arithmetic genus of the curve of singularities $C$, that is, $h^1(X,
\OO_X) = h^1(C, \OO_C)$.  In particular, when $C$ is smooth, $X$ is a
non-normal del Pezzo surface with $H^1(X,\OO_X) = 0$.  

When $C$ is wildly cuspidal, Reid shows that the
normalization $X^\nu$ is the cone over a rational curve of degree $d
\geq 1$
and the normalization morphism $\phi: X^\nu \to X$ is the contraction to
$C$ of the non-reduced double structure $D$ on a ruling $D^{\red}$
in $X^\nu$. Moreover,  the restriction of $\phi$ to the ruling
$D^{\red}$ gives a desingularization of $C$.
This construction requires the cusps of $C$ to be wild,
because
otherwise the resulting variety $X$ is not Gorenstein (cf. \cite[\S4.4]{rei1}).
Such examples $X$ are non-normal del Pezzo surfaces of
anti-canonical degree $K_X^2 = d$ and 
irregularity 
$h^1(X, \OO_X) = h^1(C, \OO_C) > 0$.  
 Reid constructs explicit surfaces $X$ where the curve $C$
has cusps of arbitrarily large order, showing that the irregularity of
a non-normal del Pezzo surface may
be arbitrarily large. 
Such surfaces are arguably the most pathological examples of del Pezzo
surfaces. 

In light of this classification, 
a scheme $\bar X$ admits a $k$-form that is a del Pezzo surface over
$k$ with irregularity $h^1(X,\OO_X) > 0$ only if $\bar X$ 
is a non-normal del Pezzo surface or $\bar X$ is a non-reduced
del Pezzo scheme.  Main Theorem~\eqref{item-main-theorem-construction}
asserts that regular 
forms can exist in either case, and Main
Theorem~\eqref{item-main-theorem-bound} provides a
numerical inequality that, in particular, rules out a large class of 
non-normal del Pezzo surfaces that could potentially admit regular
forms.

\subsection{A prior example}
Acknowledging
Reid's non-normal classification, Koll\'ar
remarks in \cite[Rem.~5.7.1]{kol1}
on the possibility that regular del Pezzo surfaces could 
have positive irregularity.
He
ultimately leaves the issue 
unresolved, although his question is repeated later by 
Schr\"oer in \cite{sch1}.  There Schr\"oer constructs an
interesting example of a normal del Pezzo
surface $Y$ in characteristic $2$ that
is a local complete intersection (l.c.i.) and regular away from one
singular point $y_\infty$, and has irregularity $h^1(Y, \OO_{Y}) = 1$. 
This variety $Y$ is a form
of the example of Reid whose normalization morphism
is described as 
the collapse of a non-reduced double line in $\P^2$
to a reduced cuspidal curve $C$ with arithmetic genus $1$.
Schr\"oer's 
method of construction is to begin with any imperfect field $k$ of
characteristic $2$ along with
a non-normal $k$-form $X$ of the variety constructed by Reid. 
Schr\"oer then studies actions of 
the infinitesimal group scheme $\alpha_2$ on $X$. He uses one such
action to twist the field of definition, thus obtaining the twisted
form $Y$ which he proves to be l.c.i.~and normal. 
Schr\"oer shows moreover that no
$\alpha_2$-twisting of the variety $X$  
can remove the singularity at $y_\infty$,
and hence his surface $Y$ is an optimal one obtainable by
this method.

\subsection{A brief outline}

The numerical bound in Main Theorem~\eqref{item-main-theorem-bound}
is obtained in \S \ref{numerical-bounds-section} by studying 
the inseparable degree $p$ covers
associated to Frobenius-killed classes in the first cohomology group
of pluri-canonical line bundles on $X$.
Such covers were
studied by Ekedahl in \cite{eke1} and shown to have peculiar
properties, which we interpret to deduce the inequality 
\eqref{inequality}.
The notion of algebraic foliation on a regular (possibly non-smooth)
variety is developed in \S
\ref{general-construction-section}, where we
extend results of Ekedahl \cite{eke2} from the
smooth case.
The surfaces $X_1$ and $X_2$ mentioned in Main
Theorem~\eqref{item-main-theorem-construction} 
are exhibited as quotients by explicit algebraic foliations on a
regular 
form of 
a non-reduced double plane in projective $3$-space
in \S\ref{geometric-construction-section}. 
We conclude in \S \ref{local-chart-section} with a detailed
study of the example $X_1$, a regular and
geometrically integral del Pezzo surface with $h^1(X_1,
\OO_{X_1}) = 1$.

\vspace{1em}
\noindent \textbf{Acknowledgements.}
I thank my thesis advisor Johan de Jong for introducing me to this
topic, and for his charitable guidance and
unceasing optimism that led me through the discovery of the contained
results.  I also would like to thank J\'anos Koll\'ar and Burt Totaro
for offering helpful comments on an earlier draft.
\\

%\newpage

\begin{center}
{\sc Notation}\\
\end{center}
%Unless otherwise stated, 

\vspace{0.25em}
\begin{itemize}
\renewcommand{\labelitemi}{$\cdot$}
\item All fields are assumed to be of characteristic $p \geq 2$.
\item A \emph{variety} {over a field $k$} is a finite-type,
  integral $k$-scheme.  
\item  $k(X)$ denotes the function field of a $k$-variety $X$.
\item $k_X := H^0(X,\OO_X)$ denotes the field of global functions of a
  proper $k$-variety $X$.
\item $K_X$ denotes the canonical divisor associated to the dualizing
  sheaf $\omega_{X}$ of a Gorenstein variety $X$.
\item $d = K_X^2$ denotes the anti-canonical degree of a del
  Pezzo surface $X$,  computed as the self-intersection number over
  the field $k_X = H^0(X, \OO_X)$. 
\item $h^i(X, \sF) := \dim_{k_X} H^i(X,\sF)$ denotes the
  dimension over the field $k_X = H^0(X,\OO_X)$ of the $i$th
  cohomology group of a sheaf $\sF$ on a proper variety $X$.
\item $q = h^1(X,\OO_X)$ denotes the irregularity of a proper surface
  $X$.
\item $\chi(\sF) := \sum_i (-1)^i h^i(X, \sF)$ denotes the Euler
  characteristic of the coherent sheaf $\sF$ on a proper variety $X$.
\item $\Frobenius_X:X \to X$ denotes the absolute Frobenius morphism of
  a scheme $X$.  
\item $\Frobenius_{X/S}$ denotes the Frobenius morphism relative
 to a morphism of schemes $X \to S$.
\item $\Omega_{Z/S}$ denotes the sheaf of relative K\"ahler
  differentials of an $S$-scheme $Z$.
\item $T_{Z/S} := \scr{H}om(\Omega_{Z/S}, \OO_Z)$ denotes the relative
  tangent bundle of an $S$-scheme  $Z$.
\end{itemize}

\section{Numerical bounds on del Pezzo surfaces with
  irregularity}\label{numerical-bounds-section}
\setcounter{equation}{0}

The goal of this section is to find a restriction on the possible
integer pairs $(d, q)$ that exist as the degree $d = K_X^2$ and 
irregularity $q = h^1(X, \OO_X)$ of a normal, l.c.i.~del Pezzo 
surface $X$ over a field $k$, under the assumption that $q \neq 0$.
Our method is to study the torsors, for certain non-reduced group
schemes $\alpha_{\sL}$,
associated to Frobenius-killed classes in the first cohomology group
of pluri-canonical line bundles $\sL := \omega_X^{\otimes m}$ on $X$.
Originally studied by Ekedahl in
\cite{eke1, eke2}, 
the existence of such torsors are often used as a technique to
work around the failing of Kodaira vanishing in characteristic $p > 0$.

\subsection{$\alpha_{\sL}$-torsors}\label{alpha-L-torsors-subsection}
\setcounter{equation}{0}

We briefly summarize here the basic properties of
$\alpha_{\sL}$-torsors, but we refer the reader to
\cite{eke1} or \cite[\S II.6.1]{kol2} for more detailed accounts.

Let $\sL$ be a line bundle on a variety $X$ over a field $k$ of
characteristic $p$ such that $H^1(X, \sL) \neq 0$.  We note that if
$\sL$ is the inverse of an ample line bundle, then this would be an
example of the Kodaira non-vanishing phenomenon.
Assume as well that pulling-back by
the absolute Frobenius morphism $\Frobenius_X: X \to X $ does not yield
an injective homomorphism from
$H^1(X, \sL)$,
that is,
there exists a nonzero class
$\bar \xi \in H^1(X, \sL)$ for which 
$$\Frobenius_X^\ast(\bar \xi)
= 0 \in H^1(X,\sL^{\otimes p}).$$
 
The Frobenius pull-back defines a surjective homomorphism of group schemes
over $X$ from $\sL$ to $\sL^{\otimes p}$.
Let $\alpha_{\sL}$ be the group scheme defined as the kernel
of this homomorphism, which by definition sits in the short exact sequence
of group schemes,
\begin{equation}\label{short-exact-seq}
  0 \to \alpha_{\sL} \to \sL \stackrel {\Frobenius_X^\ast} \to
  \sL^{\otimes p} \to 0.
\end{equation}
Locally the group scheme $\alpha_{\sL}$ is isomorphic to the
constant non-reduced group scheme $\alpha_p$, whose fibre over $X$
is the kernel of the $p$th power endomorphism of the additive
 group $\G_a$.

The long exact sequence in cohomology associated to
\eqref{short-exact-seq} shows that the class
$\bar \xi$ comes from a nonzero class $\xi \in H^1(X,
\alpha_{\sL})$ that is determined up to an element of the
cokernel of $\Frobenius^\ast_X:H^0(X, \sL) \to H^0(X,\sL^{\otimes p})$.
Via \u Cech 
cohomology, one sees that $\xi$
gives rise to a nontrivial $\alpha_{\sL}$-torsor $f: Z \to X$.
The morphism $f:Z \to X$ is  purely inseparable of degree
$p$ because $\alpha_{\sL}$ is a non-reduced finite group scheme of
degree $p$ over $X$.

To describe this $\alpha_{\sL}$-torsor more
explicitly, notice that a Frobenius-killed class $\bar \xi \in
H^1(X,\sL)$ corresponds to a non-split extension 
of vector bundles,
$$0 \to \OO_X \stackrel i \to \sE \stackrel \pi \to \sL\inv \to 0,$$
for which there is some splitting
$\sigma: \sL^{\otimes -p} \to \Frobenius_X^\ast \sE$ 
of the morphism $F_X^\ast \pi$. 
We note that the choice of splitting is determined up to
an element of $H^0(X, \sL^{\otimes p})$.
The  affine algebra $f_\ast \OO_Z$ is the quotient of
the symmetric algebra $\Sym^\ast( \sE)$ by the ideal generated by $1 -
i(1)$ as well as the image of $\sigma$ in $\Frobenius_X^\ast \sE \subseteq
\Sym^p(\sE) \subseteq \Sym^\ast(\sE)$.  Two splittings yield isomorphic
quotients precisely when they differ by an element in the
image of $\Frobenius_X^\ast:H^0(X,\sL) \to H^0(X,\sL^{\otimes p})$.  Thus we see
that this explicit construction is also determined by the data of some
class $\xi \in H^1(X, \alpha_{\sL})$ lifting $\bar \xi \in H^1(X, \sL)$.

\begin{proposition}[Ekedahl]\label{euler-characteristic-of-alpha-L-torsor-prop}
  If $X$ is a normal, projective, Gorenstein (resp. l.c.i.) variety and
  $f:Z \to X$ a 
  nontrivial 
  $\alpha_{\sL}$-torsor for some line bundle $\sL$, then $Z$ is a
  projective, Gorenstein (resp. l.c.i.) variety satisfying:
\begin{enumerate}
  \item $\omega_Z \cong f^\ast(\omega_X \otimes
  \sL^{\otimes p -1})$, 
  \item $\chi(f_\ast\OO_Z) =  \sum_{i=0}^{p-1}\chi(\sL^{\otimes -
    i})$. \label{item-euler-char}
\end{enumerate}
\end{proposition}
\begin{proof}
  Showing that $Z$ is integral with $\omega_Z \cong f^\ast(\omega_X
  \otimes \sL^{\otimes p-1})$ when $X$ is normal can be found
  in \cite[\S 1]{eke1} or \cite[Prop. II.6.1.7]{kol2}. 
  From the explicit description of $Z$ given above, we obtain
  a filtration of $f_\ast \OO_Z$, given by the images of $\Sym^i(\sE)$,
  whose successive quotients are isomorphic to $\sL^{\otimes - i}$, for $0
  \leq i < p$ 
  (cf. \cite[Prop. 1.7]{eke1});
  this immediately yields the Euler
  characteristic formula in \eqref{item-euler-char}.
  Finally, if $X$ is l.c.i., then $Z$ is too as
  it embeds in the affine bundle 
  $\mathbf{Spec}~ \Sym^\ast(\sE)/(1 - i(1))$  over $X$ as the Cartier
  divisor defined locally by $\sigma(s)$, where $s$ is a local generator
  of $\sL^{\otimes -p}.$
\end{proof}
  
  We intend to use Proposition
\ref{euler-characteristic-of-alpha-L-torsor-prop}~\eqref{item-euler-char}
to relate the Euler characteristic of the structure sheaf of a normal,
l.c.i.~del Pezzo surface 
$X$ to that of a nontrivial $\alpha_{\sL}$-torsor $f:Z \to X$.  
Yet, if the fields $k_Z := H^0(Z, \OO_Z)$ and $k_X := H^0(X,\OO_X)$
do not coincide, then the Euler characteristics $\chi(\OO_Z)$ and
$\chi(f_\ast \OO_Z)$ differ by a factor of $[k_Z:k_X]$:
$$\chi(f_\ast \OO_Z) = [k_Z:k_X]\cdot \chi(\OO_Z).$$
The following easy lemma controls 
this factor, showing it is either $1$ or $p$.

\begin{lemma}\label{global-sections-of-Z-lemma}
  If $f:Z \to X$ is a finite dominant morphism of degree $d$
  from a proper
  variety $Z$ to a normal, proper variety $X$ over $k$,
  then $k_Z := H^0(Z, \OO_Z)$ is a field extension of $k_X:=
  H^0(X,\OO_X)$ whose degree divides $d$, that is,
  $$[k_Z:k_X] ~|~ d.$$
\end{lemma}
\begin{proof}
  There are field extensions $k_X \subseteq k(X) \subseteq k(Z)$ and
  $k_Z \subseteq k(Z)$.
  Because $X$ is normal and $f$ is finite, $k_Z \cap k(X) = k_X$.
  Therefore, $k_Z \otimes_{k_X} k(X)$ is a subfield of $k(Z)$, of degree
  $[k_Z: k_X]$ over $k(X)$, and hence $[k_Z: k_X]$ 
  divides $[k(Z): k(X)] = d$.
\end{proof}

\subsection{Normal del Pezzo surfaces of local complete intersection}
\setcounter{equation}{0}

Let $X$ be a normal, l.c.i.~del Pezzo surface over a field $k$ such that
for some integer $n$ the cohomology group $H^1(X, nK_X) \neq 0$
(e.g. $X$ is a regular del Pezzo surface with irregularity and $n = 1$).
We will see that the construction of the previous subsection can be used
to create a degree $p$ inseparable 
morphism $f: Z \to X$ whose existence puts restrictions on the possible
pairs of integers $(d,q)$ that arise as the degree $d$ and
irregularity $q$ of such $X$.
 The normalcy condition is used to ensure the integrality of
$Z$, and the l.c.i.~condition guarantees that we may use the following
version of the Riemann-Roch theorem:

\begin{theorem}[Riemann-Roch]\label{riemann-roch-for-l.c.i.-surfaces-theorem}
  If $D$ be a Cartier divisor on a $2$-dimensional variety $X$ of
  local complete intersection, then
  $$ \chi(\OO_X(D)) = \chi(\OO_X) + \frac{1}{2} D.(D-K_X).$$
\end{theorem}
\begin{proof}
  The Grothendieck-Riemann-Roch theorem 
  asserts for any line bundle $\sL$ on $X$, 
  \begin{equation}\label{eqn-riemann-roch}
    \chi(\sL) = \int_X \ch(\sL) \frown (\td(T_{vir}) \frown [X]),
  \end{equation}
  where $T_{vir}$ is the virtual tangent bundle of $X$ 
  (cf.~\cite[Cor. 18.3.1(b)]{ful1}).
  The Todd class is given by $\td(T_{vir}) = 1 +
  \frac{1}{2}c_1(T_{vir}) + \frac{1}{12}(c_1(T_{vir})^2 +
  c_2(T_{vir}))$, and the Chern 
  character by $\ch(\sL) = 1 + c_1(\sL) + \frac{1}{2}c_1(\sL)^2$. 
  Taking $\sL := \OO_X$, we see that $\chi(\OO_X) =
  \frac{1}{12}\int_X (c_1(T_{vir})^2 + c_2(T_{vir})) \frown [X]$.  Substituting
  these expressions into \eqref{eqn-riemann-roch} for
  $\sL := \OO_X(D)$ results in the formula
  $$ \chi(\OO_X(D)) = \chi(\OO_X) + \frac{1}{2} \int_X  D.(D +
  c_1(T_{vir})).$$
  We finish by noting that $c_1(T_{vir}) = -K_X$, due to the adjunction
  formula for local complete intersections.
\end{proof}

The main result of this section is the following:

\begin{theorem}\label{numerical-formula-theorem}
  Let $X$ be a normal, l.c.i.~del Pezzo surface with
  irregularity $q_X = h^1(X, \OO_X)$.
  \begin{enumerate}
  \item \label{item-existence-of-m}
    If $q_X > 0$ then there exists a positive integer $m$ such that
  the line bundle $\sL := \omega_X^{\otimes m}$ has the following
  property:\\ 
  
  \noindent $(\ast)$ the absolute Frobenius pullback
  $\Frobenius_X^\ast: H^1(X,\sL) \to H^1(X,\sL^{\otimes p})$ 
  has a nontrivial kernel.\\

\item \label{item-existence-of-torsor}
  If $\sL$ is a line bundle that satisfies $(\ast)$ and is 
    numerically equivalent to $\omega_X^{\otimes m}$ for some integer
    $m$,
    then there exists a nontrivial $\alpha_{\sL}$-torsor $Z$
  that is an l.c.i.~del Pezzo surface of anti-canonical degree
  \begin{equation}\label{eqn-degree-of-Z}
    K_Z^2 = p^{1-e}(1 + m(p-1))^2K_X^2.
  \end{equation}
  The field  
  $k_Z := H^0(Z,\OO_Z)$ is an extension of $k_X := H^0(X,\OO_X)$ of
  degree $p^e$ with  $e \in \{0,1\}$, and if
  $q_Z := h^1(Z, \OO_Z)$ denotes the irregularity of $Z$, then 
  \begin{equation}\label{main-equation}
    p^e(1 - q_Z) = p - pq_X + \frac {mp(p-1)K_X^2} {12} \left ( 3 +
    m(2p-1) \right ).
  \end{equation}
  \end{enumerate}
\end{theorem}
\begin{proof}
  The existence of an integer $m$ as in \eqref{item-existence-of-m}
  is an immediate consequence of Serre's
  theorems on duality and vanishing of higher cohomology.  Let $\sL$ 
  be any line bundle satisfying the hypothesis of
  \eqref{item-existence-of-torsor}, and let $Z$ be
  any $\alpha_{\sL}$-torsor $Z$ associated to a nonzero Frobenius-killed
  cohomology class $\bar \xi \in H^1(X, \sL)$. 
  By Proposition \ref{euler-characteristic-of-alpha-L-torsor-prop},
  the torsor $Z$ is an l.c.i.~variety with dualizing sheaf $\omega_Z \cong
  f^\ast(\omega_X \otimes \sL^{\otimes 
    p  -1})$.
  Hence, we can compute the anti-canonical degree of $Z$ (over $k_Z$) as
  \begin{equation*}%\label{eqn-incomplete-degree-of-Z}
    K_Z^2  = \frac{\deg f}{[k_Z:k_X]} \cdot (1 + m(p-1))^2K_X^2.
  \end{equation*}
  Since $\deg f = p$, Lemma~\ref{global-sections-of-Z-lemma} implies
  that $[k_Z:k_X] = p^e$ with $e \in \{0,1\}$, which 
% combined with \eqref{eqn-incomplete-degree-of-Z} 
  proves \eqref{eqn-degree-of-Z}.

  Since both $\omega_X\inv$ and $\sL\inv$ are ample line bundles on
  $X$ and $f$ is a 
  finite morphism, the line bundle $\omega_Z\inv$ is ample 
  and $Z$ is therefore an l.c.i.~del Pezzo surface.
  Moreover, Proposition
  \ref{euler-characteristic-of-alpha-L-torsor-prop} 
  gives the equality
  $\chi(f_\ast\OO_Z) = \sum_{i=0}^{p-1} \chi(\sL^{\otimes -i}).$
  The Riemann-Roch theorem
  shows that $\chi(\sL^{\otimes -i})$ is independent of the numerical
  equivalence class of $\sL^{\otimes -i}$ and hence that
\begin{align*}
  \chi(\sL^{\otimes -i}) &=  \chi(\omega^{\otimes -mi}) \\
  &= \chi(\OO_X) + \frac{mi(mi+1)}{2}{K_X^2}.
\end{align*}
 We substitute this into our expression for $\chi(f_\ast \OO_Z)$ and
 use the well-established 
 formulae for summing consecutive integers and their squares to obtain
 \begin{align*}
   \chi(f_\ast\OO_Z) & = 
%p \chi(\OO_X) + \sum_{i=0}^{p-1} \frac{K_X^2}{2} im (im
%    + 1) \\
   p\chi(\OO_X) + \frac {K_X^2} 2 \sum_{i=0}^{p-1} (m^2 i^2 + mi)\\
%    & = p\chi(\OO_X) + \frac {d_X} 2 \left( \frac{mp(p-1)}2 + \frac{ m^2
%      (p-1)p(2p-1)}6\right)\\
   & = p\chi(\OO_X) +  \frac {mp(p-1)K_X^2} {12} \left ( 3 +
   m(2p-1) \right ).
  \end{align*}
Because $X$ and $Z$ are each del Pezzo surfaces, Serre duality implies
that  $H^2(X, \OO_X) = 0 = H^2(Z, \OO_Z)$.
  Therefore, $\chi(f_\ast \OO_Z) = p^e(1 - q_Z)$ and
  $\chi(\OO_X) = 1 - q_X$.
\end{proof}

%To see that this applies in our setting, we need to show that such
%integers $m$ always exist, but this is a simple
%consequence of Serre duality and Serre's theorem on the vanishing of
%higher cohomology:
%
%\begin{lemma}\label{existence-of-m-lemma}
%  If $X$ is a Gorenstein del Pezzo surface with canonical class $K_X$,
%  then $H^1(X, nK_X) = 0$ for $n \gg 0$.  Furthermore, if $H^1(X,
%  nK_X) \neq 0$ for some $n \in \Z$, then there exists an integer $m >
%  0$ and a nonzero class 
%  $\xi \in H^1(X, mK_X)$ such that $F^\ast \xi = 0 \in H^1(X, mpK_X)$,
%  where $F: X \to X$ is the absolute Frobenius morphism.
%\end{lemma}

Main Theorem~\eqref{item-main-theorem-bound} follows as an immediate
corollary:
%.  With the integers $m \geq 1$ and $ e \in \{0,1\}$ defined as in
%Theorem \ref{numerical-formula-theorem}, we have:

\begin{corollary}\label{q-and-d-inequality-corollary}
  If  $X$ is a normal, l.c.i.~del Pezzo surface of degree $d$ and irregularity
  $q >0$,  then there exists a nontrivial
  $\alpha_{\omega_X^{ \otimes m}}$-torsor, $Z$, for which $[k_Z:k_X] =
  p^e$ for some integers $m \geq 1$ and $e \in \{0,1\}$.
  Furthermore, any such integers satisfy
  \begin{align}
    q & \geq 1 - \frac{1}{p^{1-e}} + \frac{md(p-1)(3 + m(2p-1))}{12}
    \notag \\
%    & \geq \frac{md(p-1)(3 + m(2p-1))}{12} \label{e=1-geq} \\
    & \geq  \frac{d(p^2 -1)}6, \label{m=1-geq} 
%    & \geq  \frac{d}{2} \label{p=2-geq},
  \end{align}
  with equality  in \eqref{m=1-geq}
  only if $e = 1$ and $m = 1$.
\end{corollary}
\begin{proof}
  For the first inequality, use \eqref{main-equation} of
  Theorem \ref{numerical-formula-theorem} and the fact
  $q_Z = h^1(Z, \OO_Z) \geq 0$.  For the second inequality, 
  use $e \leq 1$ and $m \geq 1$. 
\end{proof}

In the case when $q_X = 1$, the values of $p,m,q_Z, K_X^2,$ and
$h^0(X,\omega_X\inv)$ are 
completely determined by that of $e \in \{0,1\}.$  Later we construct
examples of regular del Pezzo surfaces exhibiting these values
for either choice of $e \in \{0,1\}$
(cf. \S\ref{geometric-construction-section}). 
\begin{corollary}\label{numerical-case-q=1-corollary}
  If $X$ is a normal, l.c.i.~del Pezzo surface over a field
  of characteristic $p$ with irregularity $h^1(X,\OO_X) = 1$ and 
  $Z$ is 
  a nontrivial $\alpha_{\omega_X^{\otimes m}}$-torsor for an integer
  $m \geq 1$,
  then $m = 1$, $p = 2$, and the anti-canonical
  degree $K_X^2 = [k_Z: k_X] = 2^e$ for $e \in \{0,1\}$.
  Moreover,
  the cohomology group
  $H^1(Z, \OO_Z) = 0$, and for all $n \geq 1$,
  $$h^0(X,\omega_X^{\otimes -n}) = \frac{n(n+1)}{2^{(1-e)}}.$$
\end{corollary}
\begin{proof}
  If $q_X = 1$, then the right-hand side of \eqref{main-equation} is
  positive,  
  forcing $q_Z = 0$.  Thus \eqref{main-equation} simplifies to
  \begin{equation}\notag
    p^e =  \frac {mp(p-1)K_X^2} {12} \left ( 3 +
    m(2p-1) \right ). 
  \end{equation}
  As all variables are positive integers, one can quickly solve by
  brute force. 
  If $e = 0$, then  $p = 2, K_X^2 = 1,$ and $m =
  1$.
  Similarly, if $e = 1$, then $p = 2, K_X^2 =2, m =1$.
  
  If $H^1(X, \omega_X^{\otimes n}) \neq 0$
  for some $n \geq 1$, 
  then Serre's theorem on the vanishing of higher cohomology would
  show the existence of some Frobenius-killed class in $H^1(X,
  \omega_X^{\otimes N})$, for some $N \geq n$, and then Theorem
  \ref{numerical-formula-theorem}(2) and our above argument
  shows that $N = 1$.  Thus, $H^1(X, \omega_X^{\otimes n}) = 0$
  for all $n > 1$.   
  By Serre duality, $h^1(X,\omega_X^{\otimes - n})=  
  h^1(X,\omega_X^{\otimes n+1}) = 0$ for any $n \geq 1$, and 
  Riemann-Roch therefore implies 
  $h^0(X,\omega_X^{\otimes -n}) = \chi(\omega_X^{\otimes -n}) =
  \frac{n(n+1)}2 K_X^2$.
\end{proof}

\section{Algebraic foliations on regular
  varieties}\label{general-construction-section} 
\setcounter{equation}{0}
In contrast to our task in \S\ref{numerical-bounds-section} of finding numerical restrictions on the
existence of regular del Pezzo surfaces with irregularity, we begin
the dual problem of constructing explicit examples of such surfaces.  The
$\alpha_{\sL}$-torsor construction of the previous section will again be
important to us, although we shall henceforth view them from an alternative
perspective. 
Beginning with a $k$-variety $Z$, equipped with an algebraic foliation
$\sF \subseteq T_{Z/k}$, one can construct a purely inseparable 
quotient morphism $f: Z \to Z/\sF$ that
factors the relative Frobenius morphism $\Frobenius_{Z/k}: Z \to
Z\times_{k, \Frobenius_{k}} k$.
If $Z \to X$ is an $\alpha_{\sL}$-torsor, there is a natural rank $1$
foliation given by the relative tangent bundle $T_{Z/X}$ that recovers
$X$ as the quotient $Z / \sF$, for $\sF := T_{Z/X}$.
The converse does not hold as the
quotient morphism $Z \to 
Z/\sF$ for an arbitrary (rank $1$) foliation $\sF$ is not necessarily an
$\alpha_{\sL}$-torsor for any choice of line bundle $\sL$ on $Z/\sF$.
However, when $p = 2$, 
this problem does not arise, and $Z$ may
indeed be recovered from the quotient $Z/\sF$ as some
$\alpha_{\sL}$-torsor (cf. \S\ref{section-p=2}).
Ekedahl developed this theory for smooth
varieties in \cite{eke2}, and in this section we generalize
his results to the setting of regular varieties.

\begin{proposition}\label{relating-numerical-invariants-of-X-and-Z-prop}
  Let $k$ be a finitely generated field extension of a perfect field
  $\F_2$ of \mbox{characteristic $2$}.
\begin{enumerate}
\item  If $Z$ is a regular $k$-variety and
  $\sF \subseteq T_{Z/k} \subseteq T_{Z/\F_2}$ is a rank $1$ foliation on $Z$ over
  the extension $k/\F_2$,
  then the quotient  $X := Z/\sF$ is a regular $k$-variety and
  the quotient morphism 
  $f:Z \to X$ is an $\alpha_{\sL}$-torsor for some line bundle $\sL$ on $X$.
\item  Additionally, if $Z$ is a del Pezzo surface 
  and $\sF^{\otimes 2} \cong \omega_{Z}$, then the quotient $X$ is a
  regular del Pezzo surface,
  the sheaf $\sL \inv \otimes \omega_X$ is a $2$-torsion
  line bundle, and 
  the following equations hold:
\begin{enumerate}%\renewcommand{\labelitemi}{$\cdot$}
  \item $[k_Z:k_X]\cdot \chi(\OO_Z) = 2 \chi(\OO_X) + d_X$,
  \item $K_X^2 = \frac{[k_Z:k_X] \cdot K_Z^2}{8}$.
\end{enumerate}
\end{enumerate}
\end{proposition}
\noindent The fruits of our labor will  be harvested in \S
 \ref{geometric-construction-section}, when we carefully find two such
 foliations $\sF_1$ and $\sF_2$ 
 on a specific variety $Z$. The resulting quotients $X_i = Z/\sF_i$ are
 regular del Pezzo surfaces with irregularity $q = 1$, and in this
 specific case, the line bundle $\sL$ can be identified precisely
 as the dualizing sheaf $\omega_X$
 (cf.~Cor.\ref{cor-L-is-dualizing-sheaf}).

\subsection{Quotients by foliations}\label{smooth-1-foliations-subsection}
\setcounter{equation}{0}

First we generalize the definition of
a foliation on a smooth variety (cf. \cite{eke2})
to the case of a regular variety over an
% finitely generated
 imperfect field. 

\begin{definition}\label{definition-foliation}
  Let $Z$ be a regular variety over a field extension $k$ 
% finitely
%  generated field 
  of a perfect field $\F$ of characteristic $p$. 
  A \emph{foliation (over the extension $k/\F$)}
  on $Z$ is a locally free $\OO_Z$-submodule $\sF \subseteq T_{Z/k}
  \subseteq T_{Z/\F}$ preserved by the Lie
  bracket and the $p$-th power operation  (i.e. a sub-$p$-Lie
  algebra of $T_{Z/k}$) whose cokernel $T_{Z/\F}/\sF$ is locally free.
  The \emph{rank} of a foliation $\sF$ is its rank as a
  locally free $\OO_Z$-module.  
\end{definition}

This definition recovers the usual notion of a foliation
(cf. \cite{eke2}) in the
case where $Z$ is a smooth 
variety over a perfect field $k = \F$.  Our more
general definition is contrived so that when $\pi:\sZ \to \sB$ is a
morphism of varieties from a smooth variety $\sZ$ over a perfect
field $\F$, any foliation $\sF$ on $\sZ$ (in the sense of
\cite{eke2}) that
is vertical with respect to $\pi$ (i.e.~$\sF \subseteq T_{\sZ/\sB}$)
will restrict to the generic fibre of $\pi$ as a
foliation (in our sense) over the extension $\F(\sB)/\F$.

The utility of algebraic foliations comes from the fact that one can
use them to quotient varieties to obtain purely
inseparable finite morphisms:

\begin{definition/lemma}
  Let $k/\F$ be a 
%finitely generated
  field extension of a perfect field 
  $\F$ of characteristic $p$.
  If $\sF$ is a foliation over the extension $k/\F$ on a regular
  $k$-variety $Z$, then there is a $k$-variety $Z/\sF$, which we call
  the \emph{quotient} of $Z$ by $\sF$, along
  with a purely inseparable morphism $f:Z \to Z/\sF$ that factors the
  relative Frobenius morphism $\Frobenius_{Z/k}$ and is given locally by the
  inclusion of subrings $\OO_Z^p \subseteq \OO_{Z/\sF} \subseteq
  \OO_Z$, where 
  $$\OO_{Z/\sF}:= \{ f \in \OO_Z : \delta(f) = 0 ~\textrm{for all
    local derivations }~ \delta \in \sF\}.$$
\end{definition/lemma}
\begin{proof}
The construction of $Z/\sF$ is well-defined because the definition of
$\OO_{Z/\sF}$ commutes with localization, a result which
ultimately boils down to the fact that the ring of $p$th powers
$\OO_Z^p$ is killed by any derivation. 
Since $k$, in addition to $\OO_Z^p$, is killed by all
derivations in $\sF \subseteq T_{Z/k}$, the morphism $f$ 
factors the relative Frobenius morphism 
$\Frobenius_{Z/k}: Z \to Z \times_{k, \Frobenius_k} k$.  That is,
$\Frobenius_{Z/k} = g \circ f$ for a unique morphism $g: X \to
Z\times_{k,\Frobenius_k}k$.  In particular, both $f$ and $g$ are
purely inseparable morphisms.  Moreover, since $Z$ is finite-type over
$k$, the relative Frobenius morphism $\Frobenius_{Z/k}$ is a finite
morphism, 
and hence so are the morphisms $f$ and $g$.
 Since $Z$ is a finite-type over $k$, so is the base change 
$Z \times_{k,\Frobenius_k} k$ (with structure morphism given by projection
 onto the second factor). As $X$ is finite over $Z \times_{k, \Frobenius_k}
k$, it too is of finite type over $k$.
\end{proof}

For foliations
on smooth varieties over a perfect field $k = \F$, 
the following theorem of Ekedahl
provides vital information concerning the structure of the 
quotient.

\begin{theorem}[Ekedahl]\label{Ekedahl-quotient-by-foliation-theorem}
  Let $Z$ be a smooth $n$-dimensional variety over a perfect field
  $\F$.  Let $\sF \subseteq T_{Z/\F}$ be a 
   foliation of 
  rank $r$
  and $f:Z \to X := Z/\sF$ the quotient of $Z$ by this foliation.
  Furthermore denote by
  $g: X \to Z\times_{\F, \Frobenius_\F} \F$ the morphism so that
  $g \circ f = \Frobenius_{Z/\F}$ is the relative Frobenius morphism.
  Then the following hold:  
  \begin{enumerate}
  \item $X$ is a smooth $\F$-variety;
  \item $f$ and $g$ are finite flat morphisms of degrees $p^r$ and
    $p^{n-r}$, respectively;
  \item there is an exact sequence 
    $$0 \to \sF \to T_{Z/\F} \to f^\ast T_{X/\F} \to
    \Frobenius_{Z}^\ast \sF \to 0,$$
    and hence an isomorphism
    \begin{equation*}\label{dualizing-sheaf-equation}
      f^\ast \omega_{X/\F} \cong \omega_{Z/\F} \otimes (\det \sF)^{\otimes
      1 - p}.
    \end{equation*}
  \end{enumerate}
\end{theorem}
\begin{proof}
  See \cite[\S 3]{eke2}.
\end{proof}

We now partially extend this result for our applications to regular
varieties over finitely generated imperfect fields. 

\begin{proposition}\label{dualizing-sheaf-generic-fibre-proposition}
  Let $Z$ be a regular variety over a finitely
  generated field extension $k$ of a
  perfect field $\F$. 
  Let $\sF \subseteq T_{Z/k} \subseteq T_{Z/\F}$
  be a foliation of rank $r$ over the extension $k/\F$ 
  and $f: Z \to X$ the quotient of $Z$ by this foliation.
  Then the following hold: 
  \begin{enumerate}
  \item $X$ is a regular $k$-variety;
  \item $f$ is a flat morphism of degree $p^r$;

  \item there is an exact sequence
    $$0 \to \sF \to T_{Z/\F} \to f^\ast T_{X/\F} \to
    \Frobenius_{Z}^\ast \sF \to 0,$$
    and hence an isomorphism 
    \begin{equation}\label{equation-dualizing-sheaf-generic-fibre}
      f^\ast \omega_{X/k} \cong \omega_{Z/k} \otimes
      (\det \sF)^{\otimes 1 - p}.
    \end{equation}
  \end{enumerate}
\end{proposition}
\begin{proof}
  Choose a sufficiently
  large finitely generated sub-$\F$-algebra $A \subseteq k$ so
  that $Z$ descends to a finite-type integral $A$-scheme $Z_A$, $\sF$ 
  descends to a subsheaf $\sF_A \subseteq T_{Z_A/A} \subseteq
  T_{Z_A/\F}$, and the 
  fraction field of  $A$ equals $k$.  This is possible because $Z$
  is of finite-type over $k$ and $\sF$ is a submodule of
  the coherent $\OO_{Z}$-module $T_{Z/\F}$; the $\OO_Z$-module
  $T_{Z/\F}$ is coherent 
  because $Z$ is of finite-type over a finitely generated field extension
  of $\F$.
  
  It is a classical result that the regular locus of a locally
  Noetherian scheme is an open locus (cf. \cite[Thm.~24.4]{mat}).
  Since $Z = Z_A \times_A k$ is
  regular, the regular locus on $Z_A$ is a non-empty open neighborhood of
  the generic fibre $Z$, and its image in $A$ will be an open
  neighborhood $U$ of the generic point of $A$. By 
  replacing $\Spec A$ by a sufficiently small affine subset of $U$, we may
  assume that both $Z_A$ and $\Spec A$ are regular $\F$-varieties.
  Consequently, both $Z_A$ and $\Spec A$ are 
  smooth over $\F$ since $\F$ is a perfect field
  (cf. \cite[Lem.~038V]{stacks-project}).  
  
  Because $\sF$ is a foliation, $\sF \cong \sF_A
  \otimes_A k$ 
  and $T_{Z/\F}/\sF \cong 
  (T_{Z_A/\F}/\sF_A) \otimes_A k$ are locally free.  Therefore, by
  replacing 
  $\Spec A$ by an even smaller open subscheme, we may
  assume that both $\sF_A$ and $T_{Z_A/\F}/\sF_A$ are finite locally
  free $\OO_A$-modules.
  Consider the $\OO_A$-module homomorphism
  $\sF_A \otimes \sF_A \to T_{Z_A/\F}/\sF_A$ induced by the Lie
  bracket and the $\OO_A$-module homomorphism
  $\sF_A \to T_{Z_A/\F}/\sF_A$ induced by the $p$th power operation.
  Notice that 
  both of these homomorphisms are $0$ when localized at the generic point
  of $\Spec A$ precisely because of our hypothesis that $\sF$ is a
  foliation on the generic 
  fibre $Z$.
  By the upper semi-continuity of rank,  we may restrict $\Spec A$ even
  further so that these morphisms are $0$ over all of $\Spec A$, which
  is equivalent to $\sF_A$ being a foliation on the smooth variety
  $Z_A$ over $\F$.  
  
  Now, we may
  apply Theorem \ref{Ekedahl-quotient-by-foliation-theorem} to $Z_A$
  and $\sF_A \subseteq T_{Z_A/A} \subseteq T_{Z_A/\F}$
  to obtain a smooth quotient $X_A := Z_A/\sF_A$.  
  The generic fibre
  $X_A\times_A k$ is therefore a regular variety.
  Because taking quotients by foliations
  is a local operation, $Z/\sF \cong X_A \times_A k$, proving (1).  
  Assertion (2) holds by localizing the morphism $f: Z_A \to X_A$,
  which is finite and flat of degree $r$ by Theorem
  \ref{Ekedahl-quotient-by-foliation-theorem}. The exact sequence in
  (3) follows by localizing that of 
  Theorem~\ref{Ekedahl-quotient-by-foliation-theorem}(3).
  The isomorphism in (3) follows by taking the determinant of this
  sequence, which yields 
  $$f^\ast \omega_{X_A/\F}|_Z \cong \omega_{Z_A/\F}|_Z
  \otimes (\det \sF)^{\otimes  1 - p},$$
  and then applying Lemma
  \ref{dualizing-sheaf-of-generic-fibre-lemma} 
  to each of the morphisms $Z_A \to A$ and $X_A \to A$. 
\end{proof}

\begin{lemma}\label{dualizing-sheaf-of-generic-fibre-lemma}
  Let $\pi:\ex \to \sB$ be an l.c.i.~morphism of l.c.i.~varieties over a
  field $\F$.
  Let $k := \F(\sB)$ denote the 
  function field of $\sB$.  Then the dualizing
  sheaf of the generic fibre $X := \ex \times_{\sB} k$ 
  is just the restriction of the
  dualizing sheaf of $\ex$:
  $$\omega_{X/k} =  \omega_{\ex/\F}|_X.$$
\end{lemma}
\begin{proof}
  By \cite[Def. 1.5]{har1},
  we have $\omega_\pi :=
  \omega_{\ex/\F} \otimes \pi^\ast \omega_{\sB/\F}\inv$.  As $\omega_\pi$
  commutes with arbitrary base changes,
  $$\omega_{X/k} = \omega_\pi|_X = (\omega_{\ex/\F} \otimes
  \pi^\ast \omega_{\sB/\F} \inv )|_X = \omega_{\ex/\F}|_X,$$
  with the last equality justified by $\omega_{\sB/\F}$ being locally
  trivial on $\sB$.
\end{proof}

%\begin{lemma}\label{generic-fibres-are-regular-lemma}
%  If $Y \to B$ is a morphism from a smooth $k$-variety $Y$ to a variety
%  $B$ with function field $K := k(B)$, 
%  then the generic fibre
%  $Y \times_B K$ is a regular $K$-variety. 
%\end{lemma}
%\begin{proof}
%  The localization of a regular domain is a regular
%  domain.
%\end{proof}

%We now use \eqref{family-canonical-bundles-equation} to give us
%information relating the canonical bundles of $X$ and $Z$.
%
%

\subsection{Foliations in
  characteristic two}\label{section-p=2}
\setcounter{equation}{0}

A degree $p$ inseparable morphism $f:Z \to X$ is not generally an
$\alpha_{\sL}$-torsor  for any
line bundle $\sL$, even when $f$ is the morphism arising from the
quotient by a foliation on a variety $Z$.  
Luckily, when $p = 2$ this
difficulty does not arise, which allows us to apply Theorem
\ref{numerical-formula-theorem} to the proof of Proposition
\ref{relating-numerical-invariants-of-X-and-Z-prop},
 the key result used in \S\ref{geometric-construction-section} 
to construct regular del Pezzo surfaces with irregularity.

\begin{proposition}[Ekedahl]\label{Ekedahl-degree-2-cover-prop}
  Let $f:Z \to X$ be a finite morphism  of degree $p = 2$ from a
  Cohen-Macaulay scheme $Z$ to a regular variety $X$.
  Let $\sL$ be the line bundle satisfying
  $$0 \to \OO_X \to f_\ast \OO_Z \to \sL\inv \to 0.$$
  Then $Z \to X$ is an $\alpha_s$ torsor for some $s \in \Gamma(X,
  \sL)$, viewed as a section $s \in \Hom(\sL, \sL^{\otimes 2})$, where
  $\alpha_s$ is the group scheme kernel of $\Frobenius_{\sL/X} - s:
  \sL \to \sL^{\otimes 2}$.
  Moreover, if $f$ is a   purely inseparable map, then $s = 0$ and hence 
  $f:Z\to X$ is an $\alpha_{\sL}$-torsor.
\end{proposition}
\begin{proof}
  This is proven in \cite[Prop. 1.11]{eke1} for smooth $X$, although the proof
  only requires $X$ to be regular (to guarantee that $f_\ast \OO_Z$ is
  a locally free $\OO_X$-module). 
\end{proof}

We now prove the result advertised at the beginning of this section:

\begin{proof}
[Proof of Proposition \ref{relating-numerical-invariants-of-X-and-Z-prop}]
  Proposition
  \ref{dualizing-sheaf-generic-fibre-proposition} (1) proves that $X$ is
  regular, and then Theorem \ref{Ekedahl-degree-2-cover-prop}
  shows that $f:Z \to X$ indeed arises as an $\alpha_\sL$-torsor.
  Proposition
  \ref{dualizing-sheaf-generic-fibre-proposition}(3) proves
  $f^\ast\omega_{X} \cong \omega_{Z} \otimes \sF^{\otimes -1}$
  and   Proposition
  \ref{euler-characteristic-of-alpha-L-torsor-prop} gives 
  $f^\ast \sL \cong \omega_Z \otimes f^\ast \omega_X\inv$.
  It immediately follows $\sF \cong f^\ast \sL$, and 
  also $f^\ast \omega_X \cong \sF$,
  due to the hypothesis $\omega_{Z} \cong \sF^{\otimes 2}$.
  Combining these isomorphisms, we obtain $f^\ast \omega_X \cong 
  f^\ast \sL$. Since $f$ is a finite, flat surjective map of degree
  $2$, the line bundles $\omega_X$ and $\sL$ differ by a $2$-torsion
  line bundle, and hence are $\Q$-linearly equivalent.
  If $Z$ is a del Pezzo surface, then $X$ is as well because   $f$ is
  a finite, flat 
  surjective map and so $\omega_X\inv$ is
  ample if and only if $\omega_Z\inv \cong f^\ast \omega_X^{\otimes
    -2}$ is ample.
  A straight-forward application of Theorem
  \ref{numerical-formula-theorem}~(2) gives the last two claims.
\end{proof}

 \begin{corollary}\label{cor-L-is-dualizing-sheaf}
   If $Z \to X$ is the quotient of a regular del Pezzo surface $Z$ by
   a rank $1$ foliation $\sF$ over $k/\F$, a finitely generated field
   extension of a perfect field, such that $\sF^{\otimes 2} \cong
   \omega_Z$ and $h^1(X, \OO_X) = 1$, then $Z$ is an
   $\alpha_{\sL}$-torsor for $\sL \cong \omega_X$.
 \end{corollary}
 \begin{proof}
   Corollary \ref{numerical-case-q=1-corollary} guarantees that $\F$
   is of characteristic $2$.
   By Proposition \ref{relating-numerical-invariants-of-X-and-Z-prop},
   $Z$ is a nontrivial $\alpha_{\sL}$-torsor for some line bundle $\sL$ that
   differs from $\omega_X$ by a $2$-torsion line bundle.  In
   particular, $\sL$ and $\omega_X$ are numerically equivalent, and
   therefore by
   the Riemann-Roch theorem,
   $\chi(\sL) = \chi(\omega_X)$.
   Serre duality implies
   $\chi(\omega_X) = \chi(\OO_X) = 0$, because $h^1(X, \OO_X) = 1$.
   The groups $H^0(X,\sL)$ and $H^0(X,
   \sL^{\otimes 2})$ are $0$ because
   $\sL\inv$ is ample.  Therefore, $h^1(X,\sL) = h^2(X, \sL)$, and by
   Serre duality, $h^2(X, \sL) = h^0(X , \sL\inv \otimes \omega_X)$.

   If we assume $ \sL\inv \otimes \omega_X$ is a nontrivial line bundle,
   then it follows that 
   $$h^1(X,\sL) = h^0(X,  \sL\inv \otimes \omega_X) = 0,$$
   because
   any global section of a nontrivial torsion line bundle on a
   projective variety is $0$.
   On the other hand, since $Z$ is a nontrivial $\alpha_{\sL}$-torsor,
   it corresponds to a nonzero class of the cohomology group
   $H^1(X,\alpha_{\sL})$.  The long exact sequence in cohomology
   attached to the short exact sequence of group schemes
   $$0 \to \alpha_{\sL} \to \sL \to \sL^{\otimes 2} \to 0,$$
   along with the vanishing
    $H^0(X, \sL^{\otimes 2}) = 0$,
   proves that there is an
   injection $H^1(X, \alpha_{\sL}) \subseteq H^1(X, \sL)$.  This
   latter group is $0$, yet must have a nonzero class that corresponds
   to the nontrivial $\alpha_{\sL}$-torsor $Z$, demonstrating the
   absurdity of our assumption.  Therefore $\sL \inv \otimes \omega_X
   \cong \OO_X.$ 
 \end{proof}

\section{The construction of regular del Pezzo surfaces with
  irregularity}\label{geometric-construction-section}  
\setcounter{equation}{0}
In this section we construct examples of  regular del Pezzo surfaces $X$
with $h^1(X,\OO_X) = 1$.
By Corollary \ref{numerical-case-q=1-corollary}, these
surfaces can only exist in characteristic $2$ and must have
anti-canonical
 degree $K_X^2 \in \{1,2\}$.  We construct such surfaces by applying 
Proposition \ref{relating-numerical-invariants-of-X-and-Z-prop}
to an explicit regular del Pezzo surface $Z$
%(with $h^1(Z, \OO_Z) = 0$)
and foliations $\sF$ satisfying $\sF^{\otimes 2} \cong \omega_Z$.
Once constructed, it follows from
Corollary~\ref{cor-L-is-dualizing-sheaf} that $Z \to X$ is an
$\alpha_{\omega_X}$-torsor. 

\subsection{Set-up}\label{set-up-Z-subsection}
\setcounter{equation}{0}
 Let $\F_2$ be any perfect field of
characteristic $2$. 
Let $ \zee \subseteq \P^3_{\F_2} \times
\A^4_{\F_2}$ be the family of quasi-linear quadrics given by
the vanishing of the form
$Q := \alpha_0X_0^2 + \alpha_1X_1^2 + \alpha_2X_2^2 + \alpha_3X_3^2$,
where the coordinates $[X_0: X_1 : X_2: X_3]$ are projective and 
 $(\alpha_0,\alpha_1, \alpha_2, \alpha_3)$ are affine.  As a
simplification,
we sometimes omit the symbol
$\F_2$ from our notation (e.g. we write $\Omega_{\P^3}$ instead of the
more cluttered $\Omega_{\P^3_{\F_2}/\F_2}$). Let $\sI_{\zee} \subseteq
\OO_{\P^3 \times \A^4}$ denote the ideal sheaf, generated by $Q$,
that defines $\zee$ as a subscheme of $\P^3_{\F_2} \times \A^4_{\F_2}$.

 The sequence
\begin{equation}\label{cotangent-sequence-for-zee-eqn}
0 \to \OO_{\zee}(-2) \stackrel {dQ} \to 
\OO_{\zee} \otimes ( \Omega^1_{\P^3} \oplus 
 \Omega^1_{\A^4}) \to \Omega^1_{ \zee/\F_2} \to 0
\end{equation}
is exact, and since $dQ = \sum X_i^2 d\alpha_i$ is nowhere vanishing,
the cokernel $\Omega_{\zee/\F_2}$ is a rank $6$ vector 
bundle on $\zee$.
Hence, $\zee$ is a smooth
 $\F_2$-variety.
Let $\zee_U$ denote
the restriction of the family $\zee$ to the open subscheme $U
\subseteq \A^4_{\F_2}$ that complements the 
$15$ hyperplanes of 
the form $\sum_{i=0}^3 \eps_i \alpha_i = 0$ for $\eps_i \in
\{0,1\}$.
Let $Z$ be the generic fibre of $\zee_U$ over $U$,
$$Z := (\sum \alpha_i X_i^2 = 0)\subseteq
\P^3_{\F_2(\alpha_0,\ldots,\alpha_3)}.$$
The adjunction formula implies
$\omega_Z \cong \OO_Z(-2)$, and hence $Z$ is a regular del Pezzo
surface with $K_Z^2 = 
8$ and, being a hypersurface in
$\P^3_{\F_2(\alpha_0,\ldots,\alpha_3)}$, with $h^1(Z,\OO_Z) = 0$.    

To satisfy the hypothesis of Proposition
\ref{relating-numerical-invariants-of-X-and-Z-prop}, we shall
construct, for specified subfields $k \subseteq
\F_2(\alpha_0,\ldots, \alpha_3)$,
rank $1$ foliations $\OO_{Z}(-1) \cong \sF \subseteq T_{Z/k}
\subseteq T_{Z/\F_2}$ 
over the extension  $k/\F_2$.
To construct such $\sF$, we find subsheaves $\sF_{\zee} \subseteq
T_{\zee/\F_2}$  that restrict to foliations on $\zee_U$ over the
perfect field $\F_2$.
We then take $\sF$ to be the restriction of this foliation to $Z$,
that is, $\sF := (\sF_{\zee})|_{Z}$.

\subsection{Example of degree one}\label{subsection-q-1-d-1} 
\setcounter{equation}{0}

Define $\Theta_\P: \OO_{\P^3}(-1) \to
T_{\P^3}$ as the composition
$$\Theta_{\P} : \OO_{\P^3}(-1)
\stackrel {\sum X_i^2 \partial_{X_i}}{\longrightarrow}
 \sum_{i=0}^3 \OO_{\P^3}(1)\partial_{X_i} \stackrel{\phi}{\longrightarrow} T_{\P^3},$$
where
$\phi$ is the morphism coming from the Euler sequence,
\begin{equation}
0 \to \OO_{\P^3_{\F_2}} \stackrel{\sum X_i
\partial_{X_i}}{\longrightarrow} \sum_{i=0}^3
 \OO_{\P^3_{\F_2}}(1)\partial_{X_i}
\stackrel{\phi}{\longrightarrow}
T_{\P^3} \to 0.\label{euler-sequence}
\end{equation}
Let ${\sF_{\zee}}$ denote the image of 
$$\Theta_\P \oplus 0: \OO_{\zee}\otimes \OO_{\P^3}(-1) 
\to \OO_{\zee} \otimes (T_{\P^3} \oplus T_{\A^4})  \cong
T_{\P^3 \times \A^4}|_{\zee}.$$ 
Notice that, as derivations in $\sF_{\zee}$ preserve the ideal sheaf
$\sI_{\zee}$, as well as kill functions coming from $\OO_{\A^4}$, the
sheaf $\sF_{\zee}$ is contained within the subsheaf $T_{\zee/\A^4}
\subseteq T_{\P^3 \times \A^4}|_{\zee}$.

We now proceed to 
demonstrate that ${\sF_{\zee}} \subseteq T_{\zee/\F_2}$ is foliation
over $\F_2$ when restricted to $\zee_U$.   
First we prove that $\Theta_\P \oplus 0$ is injective on all fibres
over $\zee_U$.  
It
suffices to prove this injectivity after composing
with the projection $T_{\P^3\times \A^4}|_{\zee} \to \OO_\zee \otimes
T_{\P^3}$. 
In view of \eqref{euler-sequence}, this composition fails to be
injective precisely over the points where $\sum X_i^2
\partial_{X_i}$ and $\sum X_i \partial_{X_i}$ fail to span a 
$2$-dimensional subspace of the fibre of 
$\sum_{i=0}^3 \OO_{\P^3_{\F_2}}(1)\partial_{X_i}$,
which exactly constitutes
the vanishing of all $2 \times 2$ minors of the matrix: 
$$
\begin{bmatrix}
  X_0^2 & X_1^2 & X_2^2 & X_3^2 \\
  X_0 & X_1 & X_2 & X_3 \\
\end{bmatrix}.
$$
Such minors are of the form $X_iX_j(X_i + X_j)$, for $i \neq j$, and
one quickly checks that they cannot simultaneously vanish on $\zee_U$.

As ${\sF_{\zee}}$ is rank $1$, it is preserved under the Lie bracket, and the
only remaining criterion ${\sF_{\zee}}$ must satisfy is closure under
$p$th powers.   It suffices to verify this condition on a local
generator of $\sF_{\zee}$.
On the chart
$(X_{i_0}\neq 0)$, the sheaf ${\sF_{\zee}}$ is generated by the
differential operator
$$\theta_{\P} := \Theta_{\P}(\frac{1}{X_{i_0}}) = \frac 1
{X_{i_0}}\sum X_i^2 \partial_{X_i}.$$ 
If $x_i := \frac{X_i}{X_{i_0}}$ are the local affine coordinates,
then 
$$\theta_{\P} = \sum_{i \neq i_0}(x_i^2 + x_i)\partial_{x_i},$$
 because
 $X_i\partial_{X_i} =
x_i \partial_{x_i}$, for $i \neq i_0$, and
$X_{i_0} \partial_{X_{i_0}} = \sum_{i \neq i_0} x_i \partial_{x_i}$,
as can be checked by evaluation on the functions $x_j =
\frac{X_j}{X_{i_0}}$.
We now expand $\theta_{\P}^2$, taking note that all higher-order operators in the
expansion are either $0$ (e.g. $\partial_{x_i}^2 = 0$) or are nonzero
(e.g. $\partial_{x_i}\partial_{x_j}$, $i \neq j$) but
occur with
even, hence $0$, coefficient:
\begin{align*}
\theta_{\P}^2
&= \sum_{i \neq i_0} (x_i^2 + x_i)\partial_{x_i} \circ \sum_{j\neq i_0} (x_j^2 + x_j)
\partial_{x_j} \\
&=  \sum_{i \neq i_0} (x_i^2 + x_i)  \sum_{j \neq i_0} \delta_{ij} \cdot \partial_{x_j}\\
& = \theta_{\P}.
\end{align*}
Thus, ${\sF_{\zee}} \subseteq T_{\zee/\A^4}$ is a foliation on
$\zee_U$, and the restriction $\sF := (\sF_{\zee})|_Z$ is therefore a
foliation over the extension
$\F_2(\alpha_0,\ldots, \alpha_3)/\F_2$.
Let $X_1 := Z/\sF$ be the resulting quotient.

\begin{theorem}\label{X1-theorem}
   The variety $X_1$ constructed above is a
  regular del Pezzo surface over the field $H^0(X_1, \OO_{X_1}) =
  \F_2(\alpha_0,\alpha_1,\alpha_2,\alpha_3)$  with 
  irregularity $h^1(X_{1}, \OO_{X_{1}}) = 1$ and degree $K_{X_{1}}^2 = 1$.
\end{theorem}
\begin{proof}
 The variety $Z$ defined above is a regular del Pezzo surface with
$d_Z = 8$, $\chi(\OO_Z) = 1$.  Because 
$H^0(Z,\OO_Z) = \F_2(\alpha_0,\ldots,\alpha_3)$ and $X_1$ is an
 $\F_2(\alpha_0,\ldots, \alpha_3)$-variety, 
 $H^0(X_1,\OO_{X_1}) = \F_2(\alpha_0,\ldots,\alpha_3)$ as well.  
 Proposition \ref{relating-numerical-invariants-of-X-and-Z-prop}
 therefore applies with $[k_Z:k_{X_1}] = 1$, proving the theorem.
\end{proof}

\begin{remark}\label{remark-field-of-definition-X1}
  Actually, there exists a regular del Pezzo surface $X_1'$ of degree
  and irregularity $1$ defined over the subfield
  $\F_2(\alpha_0, \alpha_1,\alpha_2) \subseteq 
  \F_2(\alpha_0, \alpha_1,\alpha_2, \alpha_3)$. 
  Indeed, the closed subscheme $\zee_{U \cap \A^3} \subseteq \zee_U$
  sitting over the inclusion $U \cap \A^3 \subseteq U$
  given by $\alpha_3 = 1$ is smooth. The foliation
  $\sF_{\zee}|_{\zee_U}$ 
  restricts to a foliation on $\zee_{U \cap \A^3}$, and
  the quotient of the generic fibre by this foliation
  is the desired surface $X_1'$.
  Any subvariety $B  \subseteq U$ of dimension strictly less than
  $3$ gives rise to 
  a singular closed subscheme $\zee_{U \cap B} \subseteq \zee_U$, and
here our method breaks down.
\end{remark}

\subsection{Example of degree two}\label{subsection-q-1-d-2}
\setcounter{equation}{0}

Let $\zee_U \to U$ be the family defined in \S\ref{set-up-Z-subsection}.
We again choose ${\sF_{\zee}} \cong \OO_{\zee}(-1)$, but this time to
be the subsheaf of 
$T_{\zee/k}$
defined by the image of 
$$\Theta_\P \oplus \Theta_\A: \OO_{\zee}(-1) 
\to \OO_{\zee} \otimes (T_{\P^3} \oplus T_{\A^4})  \cong
T_{\P^3 \times \A^4}|_{\zee},$$
for $\Theta_\P = \sum X_i^2
\partial_{X_i}$ as before, and $\Theta_\A := (\sum 
X_j)\sum_k \alpha_k \partial_{\alpha_k}$.  
Again, we work locally on the chart $(X_{i_0} \neq 0)$, with
affine coordinates $x_i := \frac{X_i}{X_{i_0}}$.
A local generator of ${\sF_{\zee}}$ is given by $\theta = \theta_\P +
\theta_{\A}$, for $\theta_\P = 
\sum_{i \neq i_0}(x_i + x_i^2)\partial_{x_i}$ and $\theta_\A = (1 +
\sum_{i \neq i_0} x_i)\sum \alpha_j \partial_{\alpha_j}$.
We saw above that the image of $\Theta_{\P}$ preserves the ideal sheaf
$\sI_{\zee}$, and therefore is contained within $T_{\zee/\F_2}$.
We now check that the image of $\Theta_{\A}$ does as well.
On the chart $(X_{i_0} \neq 0)$, the ideal $\sI_{\zee}$ is generated
by 
$$q := \frac{1}{X^2_{i_0}}\cdot Q = \alpha_{i_0} + \sum_{i \neq 
  i_0} \alpha_i x_i^2. $$
As $\theta_{\A}(q) = (1 + \sum_{i \neq i_0} x_i)q = 0$,
the image of $\Theta_{\A}$ preserves the ideal sheaf, and therefore
the image of $\Theta_{\P}+\Theta_{\A}$ is contained in
$T_{\zee/\F_2}$, that is, $\sF_{\zee} \subseteq T_{\zee/\F_2}$.

We next begin to show that the subsheaf $\sF_{\zee} \subseteq
T_{\zee/\F_2}$ is a foliation over $\F_2$ on $\zee$.  
In the previous subsection, we showed that $\Theta_{\P}$ is injective on
fibres over $\zee_U$, and it immediately follows that the same is true
of the sum $\Theta_{\P} \oplus \Theta_{\A}$.
Hence ${\sF_{\zee}}$ is a subbundle of $T_{\zee/\F_2}$.
The Lie bracket preserves $\sF_{\zee}$ simply because $\sF_{\zee}$ is
rank $1$,
and as before, our final verification is whether ${\sF_{\zee}}$
is closed under squaring.  The following local calculation
shows just that:  
\begin{align*}
\theta^2 & = (\theta_\P + \theta_\A)^2 \\
& = \theta_\P^2 + \theta_\P \circ \theta_\A + \theta_\A \circ
\theta_\P + \theta_\A^2\\
& = \theta_\P + (\sum_{i \neq i_0}x_i + x_i^2)(\sum \alpha_j
\partial_{\alpha_j}) + 0 + (1 + \sum_{i \neq i_0} x_i)^2(\sum \alpha_j
\partial_{\alpha_j})\\
& = \theta_\P + \theta_\A = \theta.
\end{align*}
This proves that ${\sF_{\zee}}$ is a foliation on $\zee_U$ over $\F_2$.
Let $\sF :=(\sF_{\zee})|_Z$ be the restriction of this foliation to
$Z$.
If
$k := \F_2(\alpha_i\alpha_j: 0 \leq i,j \leq 3) \subseteq
\F_2(\alpha_0,\ldots, \alpha_3)$, then $\sF
\subseteq T_{Z/k}$, since the image of both $\Theta_{\P}$ and
$\Theta_{\A}$ kills all elements of $k$. 
Let $X_2 := Z/\sF$ be the resulting quotient $k$-variety. 

\begin{theorem}\label{X2-theorem}
  The variety $X_2$ constructed above is a
  regular del Pezzo surface over the field
  $H^0(X_2, \OO_{X_2}) = \F_2(\alpha_i\alpha_j: 0 \leq i,j \leq 3)
  \subseteq \F_2(\alpha_0,\ldots, \alpha_3)$  
  with irregularity $h^1(X_{2}, \OO_{X_{2}}) = 1$ 
  and degree $K_{X_{2}}^2 = 2$.
\end{theorem}
\begin{proof}
  We reiterate that $Z$ is a regular del Pezzo surface with $K_Z^2 =
  8, \chi(\OO_Z) = 1,$ and 
  $H^0(Z, \OO_Z) = \F_2(\alpha_0,\ldots,\alpha_3)$. 
  The variety $X_2$ is defined over the field
  $k = \F_2(\alpha_i\alpha_j : 0 \leq i , j \leq 3)$, and therefore
  $k \subseteq H^0(X_2,\OO_{X_2}) \subseteq H^0(Z, \OO_Z)$.  The
  foliation $\sF$ does not kill all of $H^0(Z, \OO_Z)$, since
  $\theta(\alpha_0) = \alpha_0(1 + \sum_{i \neq i_0}x_i) \neq
  0$. Hence, $\alpha_0$ is not contained in $H^0(X_2, \OO_{X_2})$, which is
  therefore a proper subfield of $H^0(Z,\OO_Z)$ containing $k$.  
  As $k$ is of index $2$ in $H^0(Z, \OO_Z)$, the fields
  $H^0(X_2, \OO_{X_2})$ and $k$ must
  coincide.   We conclude by 
  applying
  Proposition \ref{relating-numerical-invariants-of-X-and-Z-prop} with
  $[k_Z:k_{X_2}] = 2$.
\end{proof}

\subsection{Geometric reducedness}
\label{geometric-reducedness-section}
\setcounter{equation}{0}
We conclude this section by proving that, of our examples constructed
above,
the surface of degree $1$  is
geometrically reduced while the surface of degree $2$ is
geometrically non-reduced.   

\begin{proposition}\label{geometric-reductivity-proposition}
   The regular del Pezzo surface $X_1$ is geometrically reduced, but
   the regular del Pezzo surface $X_2$ is geometrically non-reduced.
% and its generic point
%   has geometric embedding  dimension $1$. 
\end{proposition}
\begin{proof}
  Let $k_i := H^0(X_i, \OO_{X_i})$ denote the field of global function on
  $X_i$, for $i \in \{1,2\}$.
  We will begin with the case of $i = 1$, and we will use the notation
  established in \S\ref{subsection-q-1-d-1}.  Since $X_1$ is
  Cohen-Macaulay, it is geometrically reduced if and only if it is so
  generically, and thus suffices to prove $X_1$ is geometrically reduced 
  on the affine chart over which
  $$\OO_{Z_{\bar k_1}} = \bar k_1[x_1, x_2, x_3]/ \ell^2, \hspace{1em}
  \textrm{with }\ell:=
  \sqrt{\alpha_0}  +  \sqrt{\alpha_1}\cdot x_1  + \sqrt{\alpha_2}\cdot x_2 +
  \sqrt{\alpha_3}\cdot  x_3.$$
  The ring $R:=
  \OO_{(X_1)_{\bar k_1}}$ is the subring of $\OO_{Z_{\bar k_1}}$ on which the
  differential $\theta_\P := \sum_{i \neq 0} (x_i +
  x_i^2)\partial_{x_i}$
  vanishes.

  For the purpose of proving that $R$ is reduced, assume $f \in \bar
  k_1 [x_1,x_2,x_3]$ lifts a nilpotent element of $R$.  This
  means that, in the
  ring $\bar k_1[x_1,x_2,x_3]$, the polynomial $\theta_{\P}(f)$ is divisible
  by  $\ell$,  and 
  secondly, for some $n >0$, the 
  quadratic form $\ell^2$ divides $f^n$,
  which by unique factorization implies that $f = \ell \cdot
  g$ for some polynomial $g$.  Consequently,  $\ell$ divides the
  product $\theta_{\P}(\ell) \cdot g$ due to the 
  Leibnitz rule: 
  $$ \theta_{\P}(f) = \theta_{\P}(\ell) \cdot g + \ell \cdot \theta_{\P}(g).$$ 
  We can compute $\theta_\P(\ell)$ explicitly as
  \begin{align*}
    \theta_{\P}(\ell) & = \sum_{i=1}^3 (x_i + x_i^2) \frac {\partial f}{\partial x_i}\\
    & = \ell + (\sqrt{\alpha_0} + \sum_{i=1}^3 \sqrt{\alpha_i}\cdot x_i^2).
  \end{align*}
  Consider the morphism $\bar k_1[x_1,x_2,x_3]/\ell \surject \bar k_1$
  defined by 
  $x_1, x_2 \mapsto 0$, and $x_3 \mapsto \sqrt{\alpha_0/\alpha_3}$.
  This morphism sends $\theta_{\P}(\ell) \mapsto \sqrt{\alpha_0} +
  {\alpha_0}/{\sqrt{\alpha_3}} \neq 0,$ and so $\ell$ does not divide
  $\theta_{\P}(\ell)$.  Therefore, $\ell$ must divide $g$, and hence
  $\ell^2$ divides $f$, which implies the image of $f$ in $R$ was $0$
  to begin with.  Thus $R$ is
  reduced.

  Now, consider $f: Z \to X_2$, as in \S\ref{subsection-q-1-d-2}.
  Let $k_2 := H^0(X_2, \OO_{X_2})$ and $k_2' := H^0(Z, \OO_{Z})$. 
  As the degree of the field extension is  $[k'_2:k_2] = 2$,
  and $Z$ is geometrically a first-order neighborhood of a plane,
  the variety $Z \times_{k_2} \bar k_2$ has generic point $\bar \xi_Z$
  whose local ring $k_2'(Z) \otimes_{k_2} \bar k_2$  is Artinian
  of length $4$.  If $X_2$ were geometrically reduced, then $k_2(X_2)
  \otimes_{k_2} \bar k_2$
  would be a field, and $ k_2'(Z) \otimes_{k_2} \bar k_2$ a  $2$-dimensional
  vector space over this field, with length at most
  $2$, yielding a contradiction.
\end{proof}

\section{A geometric description of the surface of degree one}\label{local-chart-section} 
\setcounter{equation}{0}
In this section we study, through explicit computation, the regular
del Pezzo surface $X_1$ over the field
$\F_2(\alpha_0,\alpha_1,\alpha_2,\alpha_3)$ constructed in
\S\ref{subsection-q-1-d-1}.
Although Remark~\ref{remark-field-of-definition-X1} asserts that there
exists an analogous regular del Pezzo surface $X_1'$ defined over the
subfield $\F_2(\alpha_0,\alpha_1,\alpha_2)$, 
for the sake of symmetry in our calculations, we will restrict our
attention to the surface $X_1$, which for convenience we will henceforth
denote by $X$.

The surface
$X$ is geometrically integral and is of anti-canonical degree and irregularity one:
$K_X^2 = 1$, $h^1(X,\OO_X) = 1$. By Reid's classification of non-normal del
Pezzo surfaces \cite{rei1}, the normalization of the geometric base change $X_{\bar k}$ 
is isomorphic to the projective plane, $X_{\bar k}^\nu \cong
\P^2_{\bar k}$, and the normalization morphism consists of the
collapse of a double line onto a
cuspidal curve $C \subseteq X_{\bar k}$ of arithmetic genus $h^1(C,
\OO_C) = 1$.  
The upshot of our calculations is a concrete realization of this
description of $X_{\bar k}$ in terms of our construction of $X$ as the
quotient by a foliation:
% (cf. \S\ref{subsection-q-1-d-1}):

\begin{proposition}\label{prop-explicit-description-of-reid}
  Let $k := \F_2(\alpha_0,\alpha_1,\alpha_2,\alpha_3)$ and $Z \to X$
  denote the quotient morphism from the regular variety 
  $Z := (\sum \alpha_i X_i^2 = 0) \subseteq \P^3_{k}$,
  defined by the foliation
  described in  \S\ref{subsection-q-1-d-1}.
\begin{enumerate}
\item  The reduced scheme
  $Z^\red_{\bar k}$ is the hyperplane  \mbox{$(\sum \sqrt{\alpha_i}X_i
    = 0) \subseteq \P^3_{\bar k},$} and  the induced morphism
  $Z^{\red}_{\bar k } \to X_{\bar k}$ is the normalization of the
  variety $X_{\bar k}$.
\item  The singular locus of $X_{\bar k}$ is a
  rational cuspidal curve $C$ of arithmetic genus one.
\item
  The inverse image of $C$ in
  $Z^{\red}_{\bar k}$ is the double line $D$
   described by 
   the equation
   $$(\sum \sqrt[4]{\alpha_i}X_i)^2 = 0.$$
\item  The cusp of $C$ sits below the unique
  point on $D$ satisfying 
  the additional equation
  $$\sum \sqrt[8]{\alpha_i }X_i = 0.$$
\end{enumerate}
\end{proposition}

\noindent This is proven in stages throughout the following subsections.

\subsection{Normalization of geometric base change}
\setcounter{equation}{0}

We recall the notation established in \S\ref{subsection-q-1-d-1}.
The variety $Z := (\sum_{i=0}^3 \alpha_i X_i^2 = 0) \subseteq
\P^3_{k}$ is a regular del Pezzo surface over the field $k :=
\F_2(\alpha_0, \alpha_1, 
\alpha_2, \alpha_3)$, and $\sF = \im(\Theta_{\P}) \subseteq T_{Z/k}$ is the
foliation on $Z$ 
over the extension $k/\F_2$ defined by $\Theta_{\P} := \sum_{i=0}^3 X_i^2
\partial_{X_i}$.  Recall that $X$ was defined as the quotient $X =
Z/\sF$ and  as before $f: Z \to X$ will denote the quotient morphism.

\begin{proposition}\label{proposition-factorization}
  The relative
  Frobenius morphism $\Frobenius_{Z/k}$ factors as
\begin{equation}\label{eqn-factoring-frobenius}
  \xymatrix{
    &
    Z \ar[d]_f \ar[ddl]_{\bar \Frobenius_{Z/k}} \ar@/^1em/[dd]^{\Frobenius_{Z/k}}\\
    & X \ar[d]_g \ar[dl]_{\bar g}\\
    (Z \times_{k,\Frobenius_k} k)^\red \ar[r] &
    Z \times_{k, \Frobenius_k} k,
}
\end{equation}
  with morphisms $\bar \Frobenius_{Z/k}$, $\bar g$, and $f$ flat and
  finite with respective degrees $8, 4,$ and $2$.
  The geometric base change of the top triangle 
  admits a further factorization,
  \begin{equation}\label{eqn-factoring-geometric-base-change}
    \xymatrix{
      Z^\red_{\bar k} \ar[r] \ar[dr]^{\bar f_{\bar k}} \ar[d]_{h}&
      Z_{\bar k} \ar[d]_{f_{\bar k}}\\
      %    \ar[ddl]_{\bar \Frobenius_{Z/k}}
      %    \ar@/^1em/[dd]^{\Frobenius_{Z/k}}\\ 
      (Z \times_{k,\Frobenius_k} \bar k)^\red &
      X_{\bar k} \ar[l]_-{\bar g_{\bar k}},\\
      %    Z \times_{k, \Frobenius_{k}} k.
    }
  \end{equation}
  where the morphism 
  $\bar f_{\bar k}: Z_{\bar k}^\red \to X_{\bar k}$ 
  identifies
  $Z_{\bar k}^\red \cong  \P^2_{\bar k}$ with the normalization
  of the variety $X_{\bar k}$. 
\end{proposition}
\begin{proof}
  Diagram \eqref{eqn-factoring-frobenius} clearly exists and commutes 
  since both $Z$ and $X$ are regular varieties and hence
  reduced schemes.
  By Proposition \ref{dualizing-sheaf-generic-fibre-proposition}, the
  morphism  $f:Z
  \to X$ is flat and finite of degree $2$. 
  %since this is just obtained from the first by base change, but $g$ is
  %not obtained by base changing the latter, since the $K$-linear
  %Frobenius morphism $F_K$ is \emph{not} simply
  %the pullback of the map $k$-linear one $F_k$.  As a matter of fact, we
  %easily see that $g$ is not flat since it factors through $\bar g$.
  %There is still hope that $\bar g$ is flat.  

  We next make computations 
  on the affine chart $U = (X_0 \neq 0)$, and by symmetry, analogous
  assertions are true over any chart of the form $(X_i \neq 0)$.
  Restricted to $U$, the top triangle of
  \eqref{eqn-factoring-frobenius} is dual to the following triangle
  of $k$-algebra morphisms: 
\begin{equation}\label{eqn-algebra-factorization}
  \xymatrix{
   &
  M :=  k[x_1, x_2,x_3]/(\alpha_0 + \sum \alpha_i x_i^2)\\
  S :=  k[u_1, u_2, u_3]/ (\alpha_0 + \sum \alpha_i u_i)\ar[ur]^{\bar
      \Frobenius_{Z/k}^\sharp} \ar[r]^-{\bar g^\sharp} &
    R := \OO_X|_U \ar[u]^{f^\sharp},
  }
\end{equation}
where $\bar \Frobenius_{Z/k}^\sharp$ is given by $u_i \mapsto x_i^2$.
It is  easy to check
that $\bar \Frobenius_{Z/k}$ is flat and finite of rank $8$
because $M$ is a rank $8$ free
$S$-module with 
basis $\langle 1, x_1, x_2,
x_3, x_1x_2, x_1x_3, x_2x_3, x_1x_2x_3 \rangle$.
Also, since $f$ is flat
and surjective, it is faithfully flat.  Thus,
$\bar g$ is flat, and hence finite of degree $8/2 = 4$. 

To finish the proof, consider the geometric base change diagram
\eqref{eqn-factoring-geometric-base-change}.  The morphism $h$,
given explicitly as a morphism between the hyperplanes
$Z^\red_{\bar k} \cong (\sum_{i= 0}^3 \sqrt{\alpha_i}X_i = 0)$ and
$(Z \times_{k, \Frobenius_k} \bar k)^\red \cong (\sum_{i=0}^3 \alpha_i
U_i = 0),$
is defined by the rule $[X_0:X_1:X_2:X_3] \mapsto
[X_0^2:X_1^2:X_2^2:X_3^2]$. This is easily seen to be a finite
dominant morphism of degree $4$.  The morphism $\bar g_{\bar k}$ is
also a dominant morphism of degree $4$ that factors $h$.
This implies that $\bar f_{\bar k}$ is finite of degree $1$, and hence
a birational morphism. 
Since $Z_{\bar k}^\red$ is a hyperplane in $\P^3_{\bar k}$, it is
isomorphic to $ \P^2_{\bar k}$, and thus $\bar f_{\bar k}$ is a
normalization morphism.
\end{proof}

\subsection{Local ring of
  functions}\label{local-ring-of-functions-subsection}
\setcounter{equation}{0}

We compute $\OO_X$ on an affine chart $(X_{i_0} \neq
0) \subseteq Z$,
but for simplicity we assume
 $i_0 = 0$, as the computations on other charts are analogous by
symmetry.

\begin{proposition}\label{affine-coordinate-ring-proposition}
  On the open $(X_0 \neq 0) \subseteq Z$, the ring of functions
  has presentation 
  $$\OO_X|_{(X_0 \neq 0)} = k[u_1,u_2,u_3, t_1, t_2,
    t_3]/(r_0,\ldots, r_6),$$ with the relations $r_i$ defined as:
$$  \begin{array}{l l l}
    r_0 := \alpha_0 + \alpha_1 u_1 + \alpha_2 u_2 + \alpha_3 u_3 &&
    r_4 := t_2t_3 + u_1u_2u_3 + (u_1+u_1^2)t_1 + u_1u_2t_2  + u_1u_3t_3  \\
    r_1 := t_1^2 + u_2u_3 + u_2u_3^2 + u_2^2u_3 &&
    r_5 := t_1t_3 + u_1u_2u_3 + u_1u_2t_1 + (u_2 + u_2^2)t_2  + u_2u_3t_3  \\
    r_2 := t_2^2 + u_1u_3 + u_1^2u_3 + u_1u_3^2 &&
    r_6 := t_1t_2 + u_1u_2u_3  + u_1u_3t_1 + u_2u_3t_2  + (u_3 + u_3^2)t_3. \\
    r_3 := t_3^2 + u_1u_2 + u_1^2u_2 + u_1u_2^2 && 
  \end{array}$$
  Moreover, the inclusion of algebras $\OO_X \subseteq \OO_Z$ dual to
  the morphism $f: Z \to X$ is given by
  $$k[u_1,u_2,u_3,t_1,t_2,t_3]/(r_0,\ldots,r_6) \to k[x_1,x_2,x_3]/(\sum \alpha_i
  x_i^2),$$
  via 
  $u_i \mapsto x_i^2$ and $t_i \mapsto x_jx_k(1 + x_j + x_k)$, for
  each assignment of indices
  $\{i,j,k\} = \{1,2,3\}$.  
\end{proposition}
\begin{proof}
Recall the diagram \eqref{eqn-algebra-factorization}, and the notation
established there.
The $S$-algebra $R = \OO_X|_{(X_0 \neq 0)}$ is flat and hence
projective as an $S$-module.  As $S$ is
isomorphic to a polynomial ring in two variables, over which all
projective modules are free,  $R$ is actually a free
$S$-submodule of rank $4$ of the free $S$-module $M$ of rank $8$ 
with basis  $\langle 1, x_1, x_2,
x_3, x_1x_2, x_1x_3, x_2x_3, x_1x_2x_3 \rangle$.

The
derivation $\theta_{\P} := \Theta_{\P}|_{(X_0 \neq 0)} = \sum_{i \neq 0} (x_i +
x_i^2)\partial_{x_i}$ is $S$-linear because $S \subseteq R =
\ker(\theta_{\P})$.
Therefore, we may compute its matrix as an $S$-module morphism 
$M \stackrel{\theta_{\P}}{\to} M$:
$$
\begin{blockarray}{r c c  c c c   c c c   c }
% header
&& 1 & x_1 & x_2 & x_3 & x_1x_2 & x_1x_3 & x_2x_3 & x_1x_2x_3 \\\\
\begin{block}{r c ( c | c c c | c c c | c )}
1   &&
    &x_1^2&x_2^2&x_3^2&        &        &        & \\
\cline{3-10}
x_1 &&
    & 1   &     &     & x_2^2  & x_3^2  &        & \\
x_2 &&
    &     & 1   &     & x_1^2  &        &  x_3^2 & \\
x_3 &&
    &     &     &  1  &        & x_1^2  &  x_2^2 & \\
\cline{3-10}
x_1x_2 &&
    &     &     &     &        &        &        & x_3^2  \\
x_1x_3 &&
    &     &     &     &        &        &        & x_2^2  \\
x_2x_3 &&
    &     &     &     &        &        &        & x_1^2  \\
\cline{3-10}
x_1x_2x_3 &&
    &     &     &     &        &        &        & 1 \\
\end{block}
\end{blockarray}
$$
This is a block matrix, which makes computing its kernel easy:

$$\begin{pmatrix}
0 & A & 0 & 0\\
0 & 1 & B & 0\\
0 & 0 & 0 & C\\
0 & 0 & 0 & 1
\end{pmatrix} \cdot
\begin{pmatrix}
  v_1 \\ 
  v_2 \\
  v_3 \\
  v_4 
\end{pmatrix} = 
\begin{pmatrix}
  Av_2\\
  v_2 + Bv_3\\
  Cv_4\\
  v_4
\end{pmatrix},
$$
and this vector equals zero if and only if $v_4 = 0, v_2 = Bv_3$ and
$AB v_3 = 0$.  In our situation, the matrix $AB = 0$, and so we see
that the kernel of $M$ defined by $v_4 = 0, v_2 = Bv_3$.  Thus, a
basis of the kernel is given by the four elements
$$\begin{array}{l l}
\phi(t_0) := &1 \\
\phi(t_1) :=& x_2x_3 + x_2^2x_3 + x_2x_3^2, \\
\phi(t_2) :=& x_1x_3 + x_1^2x_3 + x_1x_3^2, \\
\phi(t_3) :=& x_1x_2 + x_1^2x_2 + x_1x_2^2,
\end{array}$$
and so $R = k[x_1^2, x_2^2, x_3^2, t_1, t_2, t_3] \subseteq
k[x_1,x_2,x_3]/(\sum \alpha_i x_i^2)$.

Clearly, there is a surjective morphism $\phi$
 from the polynomial algebra
$k[u_1,u_2,u_3,t_1,t_2,t_3]$ onto $R$, defined by the rules $u_i
 \mapsto x_i^2$ and $t_i \mapsto \phi(t_i)$.  
 The relations $r_0, \ldots, r_6$ listed above may be verified to be
 in $\ker \phi$ simply by writing the multiplication rules for the $S$-basis
$\langle \phi(t_0), \phi(t_1), \phi(t_2), \phi(t_3) \rangle.$  
As a result, there is an induced surjective map of $S$-algebras,
$$\bar \phi: k[u_1,u_2,u_3,t_1,t_2,t_3]/(r_0,\ldots, r_6) \surject R.$$
The domain is a free $S$-module with basis $\langle 1, t_1,
t_2, t_3\rangle$, 
since all monomials in the $t_i$ can be written as
$S$-linear combinations of these elements modulo the relations $r_i$.
Therefore, $\bar \phi$ is an isomorphism.
\end{proof}

\subsection{An equation defining the singular locus}
\setcounter{equation}{0}

We  apply the Jacobian criterion to  the presentation of $R =
\OO_{X}|_{(X_0\neq 0)}$ given in Proposition
\ref{affine-coordinate-ring-proposition} to find the set of non-smooth 
points of $X$.  It turns out that these points can be described
set-theoretically as the vanishing locus of a single equation.
 \begin{proposition}\label{proposition-set-theoretic-description}
   The non-smooth locus $X^{\sing}$ of $X$ is set-theoretically equal 
   to the codimension-$1$ locus defined by the single equation $\alpha_0 +
   \alpha_1u_1^2 +  \alpha_2u_2^2 + \alpha_3u_3^2 = 0$.  In
   particular, $X$ is not geometrically normal.
 \end{proposition}
\begin{proof}
The Jacobian matrix is as
follows:
$$
\begin{blockarray}{r c  c c c   c c c }
% header
    && \partial_{u_1} & \partial_{u_2} & \partial_{u_3} &
  \partial_{t_1} & \partial_{t_2} & \partial_{t_3} \\\\
\begin{block}{r c ( c c c |  l l l ) }
r_0 &   &
\alpha_1 & \alpha_2 & \alpha_3 &\\
r_1 &&
 & u_3 + u_3^2 & u_2+ u_2^2  && \\
r_2 &&
u_3 + u_3^2 & & u_1 + u_1^2  & & \\
r_3 &&
u_2 + u_2^2 & u_1 + u_1^2 & & & \\
\cline{3-8}
r_4 && 
 & & & u_1 + u_1^2 & t_3 + u_1u_2 & t_2 + u_1u_3 \\
r_5 && 
 & (\ast) & & t_3 + u_1u_2 & u_2 + u_2^2 & t_1 + u_2u_3 \\
r_6 &&
 & & & t_2 + u_1u_3 & t_1 + u_2u_3 & u_3 + u_3^2 \\
\end{block}
\end{blockarray}
$$

As $R$ is a surface described in a $6$-dimensional affine space, the
singular locus is described by the ideal generated by the $4\times
4$-minors of this matrix.  As the Jacobian matrix comprises blocks in
the form  
$$\begin{pmatrix} 
A & 0\\
\ast & B
\end{pmatrix},$$
with $A= A^{3\times 4}$ and $B = B^{3 \times 3}$, its $4\times 4$ minors
are either the product of two $2\times 2$ minors of $A$ and $B$,
the product of an entry of $A$ by the determinant of $B$, or the product of a
$3\times 3$ minor of $A$ by an entry of $B$.  
Initially, the task of computing this ideal may appear daunting,
 but the following observation
reduces the work dramatically.

\begin{lemma}\label{lemma-minors-of-matrix}
  Let $B$ be the $3\times 3$ matrix defined above.
  \begin{enumerate}
  \item The $2 \times 2$ minors of $B$ are $0$ in $R$.
  \item The diagonal entries of $B$ generate the unit ideal in $R$.
  \end{enumerate}
\end{lemma}
\begin{proof}
  Up to cyclic
  permutations of the indices $\{1,2,3\}$, there are only two types of
  $2\times 2$-minors of $B$.
  A minor of the first type is $B_{3,3}$:
  \begin{align*}
    B_{3,3} & =
    (t_2 + u_1u_3)(u_2 + u_2^2) + (t_3 + u_1u_2)(t_1 + u_2u_3)\\
    & = u_1u_2u_3 + u_1u_2^2u_3 + (u_2 + u_2^2)t_2 + t_1t_3 + u_2u_3t_3
    + u_1u_2 t_1 + u_1u_2^2u_3\\
    & = r_5 = 0.
  \end{align*}
  A minor of the second type is $B_{1,3}$:
  \begin{align*}
    B_{1,3} & = (t_1 + u_2u_3)^2 + (u_3 + u_3^2)(u_2 + u_2^2)\\
    & = t_1^2 + u_2^2u_3^2 + u_2u_3 + u_2^2u_3 + u_2u_3^2 + u_2^2u_3^2\\
    & = r_1 = 0.
  \end{align*}
  This proves (1).

For (2), assume otherwise, and let $\mathfrak m$ be a maximal ideal
containing the ideal generated by the entries of $B$.
In the residue field $\kappa := R/\mathfrak m$, 
the image of the entry
 $u_i + u_i^2$ is $0$,
forcing $u_i = \eps_i \in \kappa$, for $\eps_i \in \{0,1\}.$
The relation $r_0 = 0$ implies $\alpha_0 + \eps_1
\alpha_1 + \eps_2 \alpha_2 + \eps_3 \alpha_3 = 0$ in
$k\subseteq \kappa$, which contradicts the algebraic independence of
the $\alpha_i$'s. 
\end{proof}

From this lemma, it follows that the ideal generated by $4 \times 4$
minors of $M$ is generated by the $3 \times 3$ minors of $A$.  
Denoting $h := \alpha_0 +
 \alpha_1u_1^2 +  \alpha_2u_2^2 + \alpha_3u_3^2$,
 these minors of $A$ are $A_0 = 0, A_1 = (u_1 + u_1^2)h, A_2 = (u_2 +
 u_2^2)h, A_3 = (u_3 + u_3^2)h$.
 Lemma~\ref{lemma-minors-of-matrix}(2) shows these 
 minors generate the principal ideal $(h)$.
\end{proof}

\subsection{The geometry of the singular locus}
\setcounter{equation}{0}

Let $C$ be the reduced subscheme corresponding to the non-smooth locus
$X^{\sing}_{\bar k} \subseteq X_{\bar k}$. 
By Proposition \ref{proposition-set-theoretic-description}, the curve
$C$ is set-theoretically cut out by the
equation $\alpha_0 + \alpha_1u_1^2 + \alpha_2u_2^2 + \alpha_3u_3^2 = 0$,
which is simply the square of the equation $\sqrt h = 0$ for
$$\sqrt{h} := \sqrt{\alpha_0} +
\sqrt{\alpha_1}u_1 + \sqrt{\alpha_2}u_2 + \sqrt{\alpha_3}u_3.$$
We expect this equation to be insufficient to
describe $C$ scheme-theoretically, because $X$ is not
smooth along this locus, so the maximum ideal of the local ring
$\OO_{X, C}$ 
requires more than one generator. 
This is indeed the case, and the structure of $C$ is as follows:

\begin{proposition}\label{proposition-scheme-theoretic-C}
  The curve $C$ is isomorphic to a rational cuspidal curve
  with $h^1(C, \OO_C) = 1$.  The singular point of
  the curve sits below the point in $Z_{\bar k} \subseteq \P^3_{\bar
    k}$ described by the intersection 
  of the three planes $(\sum \alpha_i^{1/2^j}X_i = 0)$, for $j = 1,2,3$.
\end{proposition}

\begin{proof}[Proof of \ref{proposition-scheme-theoretic-C}]
  Again we work over the chart $(X_0 \neq 0)$, and by symmetry our
  results will carry over to other opens $(X_i \neq 0)$.  We must compute 
  the quotient of the ring $R_{\bar k}/(\sqrt h)$ by its nilradical ideal.
  $$\bar k[u_1,u_2,u_3,t_1,t_2,t_3]/ (\sqrt{h}, r_0, r_1,\ldots, r_6).$$
  The first two relations $r_0 = \alpha_0 + \sum_{i=1}^3 \alpha_i u_i$
  and $\sqrt{h} = \sqrt \alpha_0 + \sum_{i = 1}^3 \sqrt{\alpha_i} u_i$
  are $\bar k$-linearly independent relations.
  Therefore, in the ring $R_{\bar k}/(\sqrt h)$,
  we can solve
  for $u_2$ and $u_3$ in terms of $u_1$, and rewrite 
  $$R_{\bar k}/(\sqrt h) \cong \bar k[u,t_1,t_2,t_3]/(r_1,\ldots,
  r_6),$$
  where the variable  $u_1$ is
  replaced by $u$ 
  and the variables $u_2$ and $u_3$ are replaced by the following expressions
  in $u$:
  $$\begin{array}{l l}
    u_2(u) = \frac{(\alpha_0 +
        \sqrt{\alpha_0\alpha_3}) + (\alpha_1 + \sqrt{\alpha_1\alpha_3}) u}{\alpha_2 + \sqrt{\alpha_2\alpha_3}}, &
    u_3(u) = \frac{\alpha_0 + \alpha_1u + \alpha_2u_2(u)}{\alpha_3}.
  \end{array}$$
  Hence the relations $r_1, r_2,$ and $r_3$ read
$$\begin{array}{l l l}
  r_1 = & t_1 ^2 + c_{10} & + ~ c_{11}u + c_{12}u^2 + c_{13}u^3 \\
  r_2= & t_2^2 +  0      & + ~ c_{21}u + c_{22}u^2 + c_{23}u^3 \\
  r_3= & t_3^2 +  0      & + ~ c_{31}u + c_{32}u^2 + c_{33}u^3, \\
\end{array}$$
for explicitly determined coefficients $c_{ij} \in \bar k$ whose
concrete description, 
for the sake of exposition, will be omitted but made available in an
auxiliary file on the author's homepage.
When written explicitly, it is straight-forward to 
check that these coefficients, for any pair $i, j \in \{1,2,3\}$,
satisfy the following relation:
\begin{equation}\label{equation-coefficient-relation}
c_{i1}c_{j3}+c_{j1}c_{i3} = 0.
\end{equation}
Make the following change of variables
$$\begin{array}{l l l}
s_1 := t_1 + \sqrt{c_{10}} + \sqrt{c_{12}}u\\
s_2 := t_2 + \sqrt{c_{22}}u\\
s_3 := t_3 + \sqrt{c_{32}}u,
\end{array}
$$
so that the relations $r_1,r_2,r_3$ become
$$\begin{array}{l l}
  r_1=& s_1^2 +
  c_{11}u + c_{13}u^3 \\   
  r_2=&  s_2^2 +  c_{21}u +  c_{23}u^3 \\
  r_3=&  s_3^2 +  c_{31}u + c_{33}u^3. 
\end{array}$$
The relations \eqref{equation-coefficient-relation}
imply $s_2^2 = \frac{c_{21}}{c_{11}} s_1^2$ and
$s_3^2 = \frac{c_{31}}{c_{11}} s_1^2$, so the 
nilradical of $R_{\bar k}/(\sqrt{h})$ must contain the relations  
$r_2' := s_2 + \frac{\sqrt{c_{21}}}{\sqrt{c_{11}}} s_1$ and
$r_3':= s_3 + \frac{\sqrt{c_{31}}}{\sqrt{c_{11}}} s_1$.  
By setting $s := s_1$, we obtain an isomorphism
$$R_{\bar k}/(\sqrt{h},r_2', r_3') \cong \bar k[u,s]/(s^2 + u(c_{11} +
c_{13}u^2),r_4,r_5,r_6).$$ 
Since $\bar k[u,s]/(s^2 + u(c_{11} + c_{13})u^2)$
is an
integral domain of dimension $1$, 
the relations $r_4, r_5, r_6$
are already $0$ in this ring.  Thus,
$$\OO_C|_{(X_0\neq 0)} 
= \bar k[u, \sqrt{u} (u + \sqrt{c_{11}/c_{13}})] 
\subseteq \bar k[\sqrt u].$$
\noindent From this description, is is clear that the only singular
point of $C$ 
is an ordinary cuspidal  
singularity of (wild) order $2$ occurring at
$${X_1^2}/{X_0^2} = u = \sqrt{c_{11}/c_{13}}.$$
Moreover, one can verify that $\sqrt[4]{c_{11}/c_{13}} = \det(A_1)/\det
(A_0)$
where the matrix $A_1$ is defined by replacing the first column of the
following matrix $A$ by the vector $b$:
$$
A := \begin{pmatrix}
  \sqrt{\alpha_1} & \sqrt{\alpha_2} & \sqrt{\alpha_3} \\
  \sqrt[4]{\alpha_1} & \sqrt[4]{\alpha_2} & \sqrt[4]{\alpha_3}\\
  \sqrt[8]{\alpha_1} & \sqrt[8]{\alpha_2} & \sqrt[8]{\alpha_3}\\
\end{pmatrix}, \hspace{2em}
b := \begin{pmatrix}
\sqrt{\alpha_0}\\ \sqrt[4]{\alpha_0} \\ \sqrt[8]{\alpha_0} 
\end{pmatrix}.
$$
Cramer's rule, 
implies that the cusp of $C$
sits below the intersection of 
the 3 planes 
$$(\sum \sqrt[2^j]{\alpha_i} X_i =
0) \subseteq \P^3_{\bar k},\textrm{ for } j=1,2,3.$$  By symmetry, this is the
only singular point of $C$.
\end{proof}

\section{Future research directions}
\setcounter{equation}{0}
 \renewcommand{\theequation}{\arabic{section}.\arabic{equation}}

\begin{question}
  Are there regular del Pezzo surfaces with positive irregularity
  in higher characteristic, that is, for $p \geq 3$? 
\end{question}

The inequality $q \geq \frac {d (p^2 - 1)}{6}$ of \eqref{inequality}
relating the degree and irregularity becomes stronger as the
characteristic grows, but it does not rule out the existence of such
surfaces in any given characteristic.  However, the author
would find it surprising if examples exist in characteristic $p \geq
5$.  

\begin{question}\label{question-smaller-insep-degree}
  Are there regular del Pezzo surfaces $X$ with positive irregularity
  over fields $k_X = H^0(X, \OO_X)$ of inseparable degree $[k_X:k_X^p]
  \leq p^2$? 
\end{question}

As pointed out in Remark \ref{remark-field-of-definition-X1}, the
geometrically integral example $X_1$ may be constructed in
characteristic $2$ over a field of 
inseparable degree $2^3$, and the geometrically non-reduced example
was constructed over a field of inseparable degree $2^4$.
The case $[k_X:k_X^p] = p$ directly addresses a question of Koll\'ar
concerning $3$-fold contractions \cite[Rem.~1.2]{kol0}.

\begin{question}
  What is the geometry of the reduced structure on the geometric base
  change of the  example  $X_2$ constructed in \S\ref{subsection-q-1-d-2}?
\end{question}

Presumably, one could explicitly compute local presentations of the
ring of regular functions on
$(X_2)_{\bar k}^\red$, as we did for the example $X_1$ in
\S\ref{local-chart-section}.  This is left as an open exercise.

\bibliographystyle{abbrv}
\bibliography{del_pezzo}

\vspace{8em}
\begin{center}
\line(1,0){250}\\
\vspace{0.5em}
{\it
Zachary Maddock\\
Department of Mathematics\\
Columbia University\\
Room 509, MC 4406\\
2990 Broadway\\
New York, NY 10027\\
maddockz@math.columbia.edu }\\
\vspace{0.5em}
\line(1,0){250}
\end{center}

\end{document}